\documentclass[10pt, a4paper, reqno]{amsart}

\usepackage{amsmath, amssymb}
\usepackage{stmaryrd}
\usepackage{enumitem}
\usepackage[all]{xy}
\usepackage{extarrows}
\usepackage{mathtools}
\usepackage{dynkin-diagrams}
\usepackage{caption}
\usepackage[super]{nth}
\usepackage{csquotes}
\usepackage{soul}
\usepackage[french, english]{babel}

\usepackage{hyperref}

\usepackage[normalem]{ulem}

\newtheorem{theorem}{Theorem}[section]
\newtheorem{corollary}[theorem]{Corollary}
\newtheorem{lemma}[theorem]{Lemma}
\newtheorem{lemmadef}[theorem]{Lemma and Definition}
\newtheorem{prop}[theorem]{Proposition}

\newtheorem*{claim}{Claim}
\theoremstyle{definition}
\newtheorem{definition}[theorem]{Definition}
\newtheorem{notation}[theorem]{Notation}

\theoremstyle{remark}
\newtheorem{example}[theorem]{Example}
\newtheorem{remark}[theorem]{Remark}

\theoremstyle{theorem}
\newtheorem{maintheorem}{Theorem}

\def\id{\mbox{Id} }

\DeclareMathOperator{\GL}{GL}
\DeclareMathOperator{\SL}{SL}

\DeclareMathOperator{\pr}{pr}

\DeclareMathOperator{\rank}{rank}

\DeclareMathOperator{\kk}{\textbf{k}}

\DeclareMathOperator{\Q}{\mathbb{Q}}

\DeclareMathOperator{\tr}{tr}

\newcommand{\tor}{tor}
\newcommand{\utor}{u-tor}

\renewcommand{\id}{\mathrm{id}}

\def\RR{\mathbb{R}}
\def\PP{\mathbb{P}}
\def\AA{\mathbb{A}}
\def\ZZ{\mathbb{Z}}
\def\CC{\mathbb{C}}
\def\QQ{\mathbb{Q}}
\def\GG{\mathbb{G}}

\def\NN{\mathbb{N}}

\DeclareMathOperator{\ant}{ant}
\DeclareMathOperator{\aff}{aff}

\newcommand{\aquot}{/ \! /}
\newcommand{\set}[2]{\{\,#1 \ | \ #2\,\}}
\newcommand{\Bigset}[2]{\left\{\,#1 \ \Big| \ #2\,\right\}}

\begin{document}	

\begin{abstract}
	We describe the Zariski-closure of sets of torsion points in connected algebraic groups. This is a generalization of the Manin-Mumford conjecture for commutative algebraic groups proved by Hindry ~\cite{Hi1988Autour-dune-conjec}.  
	He proved that every subset with Zariski-dense torsion points is the finite union of torsion-translates of algebraic subgroups.
	We formulate and prove an analogous theorem for arbitrary connected algebraic groups.
	 
	 We also define a canonical height on connected algebraic groups that coincides with  a Néron-Tate height  if $G$ is a  (semi-) abelian variety. This motivates a generalization of the Bogomolov conjecture to arbitrary connected algebraic groups defined over a number field. We prove such a generalization as well.    
\end{abstract}
	
\title[On the Manin-Mumford Theorem for algebraic groups]
{On the Manin-Mumford Theorem for algebraic groups}
\author{Harry Schmidt \and Immanuel van Santen}
\address{Harry Schmidt, University of Basel, Department of Mathematics and Computer Science, Spiegelgasse $1$, CH-$4051$ Basel, Switzerland}
\email{harry.schmidt@unibas.ch}
\address{Immanuel van Santen, University of Basel, Department of Mathematics and Computer Science, Spiegelgasse $1$, CH-$4051$ Basel, Switzerland}
\email{immanuel.van.santen@math.ch}

	
\maketitle

\tableofcontents

\section{Introduction}
We work over an algebraically closed field $\kk$ of characteristic zero. All varieties, morphisms, algebraic groups
and group actions are defined over $\kk$ and we work exclusively with the Zariski topology.

Let $G$ be an algebraic group and denote by $G_{\tor}$ the set of torsion points of $G$, i.e.~the subset of
$g \in G$ such that $g^n$ is the identity for some integer $n > 0$. It is a fundamental problem to describe subvarieties
in $G$ such that the torsion points lie dense in them. This problem is commonly associated with the Manin-Mumford conjecture for $G$ an abelian variety and was resolved by Raynaud \cite{Raynaudcourbes,Raynaudsous} for $\kk = \CC$. For a survey on the Manin-Mumford conjecture, see~\cite{Tz2000The-Manin-Mumford-}. 
A generalization for any commutative connected algebraic group $G$ is due to Hindry:

\begin{theorem}[{\cite[Th\'eor\`eme~2]{Hi1988Autour-dune-conjec}}]
	\label{Thm.MM}
	Let $G$ be a commutative connected algebraic group and let $X \subseteq G$ be a subset
	such that the torsion points $X \cap G_{\tor}$ lie dense in $X$. Then there exist finitely many
	connected algebraic subgroups $S_1, \ldots, S_r$ and $t_1, \ldots t_r \in G_{\tor}$ such that the closure of $X$ in $G$ is given by
	\[
		\overline{X} = \bigcup_{i=1}^r t_i S_i
	\]
	and $(S_i)_{\tor}$ is dense in $S_i$ for $i =1, \ldots, r$.
\end{theorem}

One of the aims of this article is to generalize the above statement to any connected algebraic group. 
This follows the desire to translate results on commutative algebraic groups to not necessarily commutative  algebraic groups. A recent example is a non-commutative version 
of the famous sub-space theorem by Yafaev \cite[Theorem~1.2]{andreisubspace}. Recently, Corvaja, Rapinchuk, Ren and Zannier have used a generalization of Manin-Mumford for the algebraic torus, namely Laurent's theorem, in order to prove deep results on boundedly generated subgroups of $\GL_n$ \cite[Theorem~1.1]{boundedlygen}. Laurent's theorem describes the Zariski-closure of subsets of finite rank subroups of an algebraic torus.  Another example that is closer to our objective is the classification of finite subgroups of $\GL_n(\mathbb{Q})$ of maximal order \cite{maximalorders, feitorders} that is related to the uniform boundedness conjecture and  Mazur's classification of finite subgroups of the rational points of elliptic curves \cite[Theorem~8]{Mazurrational}. 

However, the following
example shows that a straightforward generalization of Theorem \ref{Thm.MM} is already impossible for $\SL_2(\kk)$:

\begin{example}
	\label{Exa.Introduction}
	Let $G = \SL_2(\kk)$. The trace morphism $\tr \colon G \to \AA^1$ is constant on conjugacy classes
	and the set of torsion points is given by
	\[
		G_{\tor} = 
		\{
			\pm \begin{psmallmatrix}
				1 & 0 \\ 0 & 1
			\end{psmallmatrix}
		\} \cup 
		\tr^{-1}(S) \, , \quad \textrm{where} \
		S = \set{t+t^{-1}}{\textrm{$t \in \kk \setminus \{\pm1\}$ is a root of unity}} \, .
	\]		
	Let $X$ be	an irreducible closed hypersurface of $G$ such that $\tr(X)$ is dense in $\AA^1$.
	Then there is an open dense subset $U \subseteq \AA^1$ such that 
	$\tr^{-1}(u) \cap X$ is a (possibly reducible) curve for all $u \in U$. Since $U$ intersects $S$ in infinitely many points,
	the set of torsion points $X \cap G_{\tor}$ contains infinitely many pairwise disjoint curves and thus lies dense in $X$.
	
	Every connected $2$-dimensional subgroup $S_1$ of $G$ is conjugated to the subgroup
	of upper triangular matrices in $G = \SL_2(\kk)$ and in particular $t_1 S_1$ is  isomorphic as a variety 
	to $\AA^1 \times (\AA^1 \setminus \{0\})$
	for every (torsion) point $t_1 \in G$. However, there are many (smooth) 
	irreducible hypersurfaces $X \subseteq \SL_2(\kk)$ such that $\tr(X)$
	is dense in $\AA^1$, but  $X$ is non-isomorphic to $\AA^1 \times (\AA^1 \setminus \{0\})$: E.g., if $X$
	is the image under any automorphism of the underlying variety of $\SL_2(\kk)$ applied to the image of the closed embedding 
	$\AA^2 \to \SL_2(\kk)$, $(s, t) \mapsto 
	\begin{psmallmatrix}
		1 & s \\ t & 1+st
	\end{psmallmatrix}
	$.
\end{example}

The above example shows also that we cannot hope for such an easy description of the closed subvarieties of an algebraic group with a dense set of torsion points as in Theorem~\ref{Thm.MM}.
However, we may describe such  closed subvarieties ``up to conjugation''. 
If $X \subseteq G$ is a subset of an algebraic group $G$, then $C_G(X)$ denotes the union of all conjugacy classes of elements in $X$. 
Our first main result is a generalization of Theorem~\ref{Thm.MM} to arbitrary connected algebraic groups $G$:

\begin{maintheorem}[see Theorem~\ref{Thm.MMalg}]
	\label{Mainthm.MMalg}
	Let $G$ be a connected algebraic group and let $X \subseteq G$ be a subset such that 
	$X \cap G_{\tor}$ is dense in $X$. Then there exist finitely many commutative connected
	algebraic subgroups $S_1, \ldots, S_r$ and $t_1, \ldots, t_r \in G_{\tor}$ such that the closure of $C_G(X)$ in $G$ is given by
	\[
		\overline{C_G(X)} = \bigcup_{i=1}^r \overline{C_G(t_i S_i)} \, ,
	\]
	$t_i$ commutes with $S_i$ and $(S_i)_{\tor}$ is dense in $S_i$ for all $i=1, \ldots, r$.
\end{maintheorem}

Morally, we replace the points in a commutative algebraic group by conjugacy classes in an arbitrary algebraic group. 
For an affine reductive algebraic group $G$ this also gives the strategy of the proof; more precisely: Denote 
by $T \subseteq G$ a maximal torus. We may assume that $X$ is closed and stable under $G$-conjugation. Then one
can show that
the torsion points $X \cap T_{\tor}$ lie dense in $X \cap T$ by
using the categorical quotient $G \to G \aquot G$ with respect to conjugation. 
Thus we may apply Theorem~\ref{Thm.MM} to the closed
subvariety $X \cap T$ of the torus $T$ and apply the operator $\overline{C_G(\cdot)}$ in order to get the result.

It is far from obvious to determine $C_G(X)$, when $X$ is given and to decide whether it is a finite union of conjugacy classes of translates of commutative algebraic groups. An example, where this is done for certain matrix polynomial is \cite[Theorem 1]{ostafeerratum}. 

\smallskip

A drawback of the subvarieties $\overline{C_G(t S)}$ of $G$ from Theorem~\ref{Mainthm.MMalg} is that
the intersection of two such subvarieties is in general not a finite union of such subvarities, see \S\ref{Subsect.Intersection_SUV},
in contrast to the torsion-translates of algebraic subgroups in a commutative algebraic subgroup in Theorem~\ref{Thm.MM}, see Lemma~\ref{Lem.Subtori_I}.
 This failure can be fixed by replacing ``torsion'' $G_{\tor}$ with the coarser notion of 
``unipotent-torsion'' $G_{\utor}$, 
i.e.~$G_{\utor}$ consists of those 
$g \in G$ such that $g^n$ is unipotent for some integer $n > 0$, see Definition~\ref{Def.utorsion}.
Our second main result is a variation of Theorem~\ref{Mainthm.MMalg} to this ``unipotent-torsion"-setting.

For the formulation, we need the morphism $\pi \colon G \to \hat{G} \aquot \hat{G}$, which is the composition of the  quotient 
$G \to \hat{G}$ by the unipotent radical of $G$ followed by the categorical quotient
$\hat{G} \to \hat{G} \aquot \hat{G}$ with respect to conjugation, see~\S\ref{Subsect.Structure_results} and \S\ref{Subsect.Conjugation}.

\begin{maintheorem}[see Theorem~\ref{Thm.MMalg_unipot}]
	\label{Mainthm.MMalg_unipot}
	Let $G$ be a connected algebraic group and let $X \subseteq G$ be a subset such that 
	$X \cap G_{\utor}$ is dense in $X$. Then there exist commutative connected algebraic subgroups
	$S_1, \ldots, S_r$ and $h_1, \ldots, h_r \in G_{\utor}$ such that
	\[
		\pi^{-1}(\overline{\pi(X)}) = \bigcup_{i=1}^r  \pi^{-1}(\pi(h_i S_i)) \, ,
	\]
	where $h_i$ commutes with $S_i$ for all $i = 1, \ldots, r$. 
\end{maintheorem}

The subvarieties $\pi^{-1}(\pi(h S))$, where $S$ is a connected commutative algebraic subgroup of $G$ and $h \in G_{\utor}$
	commutes with $S$ (as in Theorem~\ref{Mainthm.MMalg_unipot}) we call \emph{unipotent-special}, see Definition~\ref{Def.uSUV}.
One can in fact show that the unipotent-special subvarieties are 
irreducible, closed in $G$ and the subset of unipotent-torsion points lie dense in them.
Moreover, the intersection of two unipotent-special subvarieties is the finite union of unipotent-special subvarieties
(see Proposition~\ref{Prop.uSUV}). This makes them good candidates for being ``special subvarieties" analagous to special subvarieties  in the unlikely intersections and the conjecture of Zilber-Pink. For more details we refer the reader to the work of Barroero and Dill, \cite[Theorem 4.1]{barroero2021distinguished}. However we immediately point out that unipotent-special subvarieties do not fit into the framework of distinguished categories, because of the lack of final objects (Axiom (A3)). 

We note that the unipotent-special subvarieties generally are not connected components of algebraic subgroups.
However, in certain situations one can give an explicit description of them. As an example we consider $G = \SL_2(\kk)$:
Here the unipotent-special subvarities different from $\SL_2(\kk)$ 
are given as follows: For every root of unity $t$ the set of 
matrices in $\SL_2(\kk)$ such that  its trace is equal to $t + t^{-1}$ form a unipotent special subvariety.
The unipotent-special subvarities that contain a dense set of torsion points are those whose trace is not equal to $\pm 2$.
One can show with elementary computations that none of these unipotent-special subvarieties is 
a connected component of an algebraic subgroup of $\SL_2(\kk)$.

\smallskip

The proofs of Theorem \ref{Mainthm.MMalg} and \ref{Mainthm.MMalg_unipot} are both by reduction to the commutative case, 
i.e.~to Theorem~\ref{Thm.MM}. 
More precisely, for Theorem~\ref{Mainthm.MMalg} we first consider the case $F \times L $, where $F$ is commutative and $L$ is affine and reductive, see Proposition~\ref{Prop.MMProduct}. This case can be done essentially via intersecting  the subset under consideration with $F \times T$,
where  $T$ is a maximal torus in $L$, similar to the affine reductive case
described directly after Theorem~\ref{Mainthm.MMalg}. The first serious difficulty arises when we replace $L$ by an arbitrary affine algebraic group $H$ (see Proposition~\ref{Prop.MMProduct_gen}), 
because no categorical quotient with respect to conjugation exists.
Moreover, the torison points do not behave well under the quotient by the unipotent radical $R_u(H)$ of $H$ 
(see e.g.~Example~\ref{Exa.Torsion_in_Borel}).
We overcome these obstacles by using a certain stratified version of a geometric quotient with respect to $R_u(H)$-conjugation on $H$, see Lemma~\ref{Lem.Stratification}. In order to get the theorem in the general case (Theorem~\ref{Thm.MMalg}) 
we then use the so-called Rosenlicht decomposition of an algebraic group 
$G$:~there is a canonical homomorphism $G_{\ant} \times G_{\aff} \rightarrow G$ (which might have a 
non-trivial kernel), where $F = G_{\ant}$ is a connected, commutative and $H = G_{\aff}$ is a connected, affine
algebraic subgroup of $G$.
Using the solution to the problem for $F \times H$ provides us then with a solution for $F H = G$.
 For Theorem~\ref{Mainthm.MMalg_unipot}, the proof turns out to be very analogous but the framework is much 
 more natural. Since every closed irreducible subset $X \subseteq G$ that dominates $\hat{G} \aquot \hat{G}$ has dense 
 unipotent-torsion points (see Corollary~\ref{Cor.X_utor_dense}), there is no hope to describe all closed subsets of $G$ with dense unipotent-torsion points.
 In particular, we are forced to divide by the unipotent radical of $G$ and thus this quotient
 gives no obstacle by design of the setup.
 
\bigskip
 
%

For the rest of this introduction, we assume that our ground field $\kk$ is equal to the field of algebraic numbers 
$\overline{\QQ}$.
Another result that is related to a translation from the commutative to the non-commutative setting 
is Breuillard's height gap theorem \cite{breuillardgap} that may be viewed as an analogue of Lehmer's conjecture. He constructs a height on finite sets of matrices that, if restricted to singletons, almost coincides with a height coming from the operator norm \cite{talamancaheight}  (see \cite[Remark 2.20]{breuillardgap}). This in turn coincides with a  Weil-height if restricted to a torus (see \cite[\S1]{breuillardgap}). In this article, we  construct a canonical height on $G$ that is strongly related to Breuillard's height (see \eqref{hG}), if $G$ is affine and with the quadratic Néron-Tate height if $G$ is an abelian variety.
Moreover, it is invariant under $G$-conjugation. 
It has the further upside that it vanishes exactly on the unipotent-torsion points of $G$ (see Lemma~\ref{Lem.Height_u_tor}). 
A knee-jerk reaction is to ask whether, going beyond Theorem~\ref{Mainthm.MMalg_unipot} a Bogomolov-type result, such as \cite{ullmobogomolov, zhangbogomolov} is true with respect to this height. This is our next result. 
\begin{maintheorem}\label{bogomolov}  Let $G$ be a connected algebraic group defined over $\overline{\Q} $. Let $X \subseteq G$, be  an irreducible subset  such that $\pi^{-1}(\overline{\pi(X)})$ is not unipotent-special. 
	Then, for a canonical height $\hat{h}_G$ on $G$, there exists an $\epsilon  = \epsilon(X)> 0$, such that the set 
\[
	\set{g \in \pi^{-1}(\overline{\pi(X)})}{\hat{h}_G(g) < \epsilon}
\]
is not dense in $\pi^{-1}(\overline{\pi(X)})$. 
\end{maintheorem}

For the exact definition of the height $\hat{h}_G$ 
we refer the reader to Definition~\ref{defheight1}, ~\ref{inv1} and~\ref{defheight2}. 
This canonical height also satisfies a certain Northcott property in the sense that a set of semi-simple elements of bounded canonical height and degree over $\Q$ is contained in  finitely many conjugation classes of semi-simple elements. 

Our proof of Theorem \ref{bogomolov} goes by reducing to the commutative case and applying the main theorem in \cite{davidphilippon} (cf. also~\cite[Proposition~6.1]{kuehnesmall}). The main new input for Theorem C is the  canonical height constructed above and the observation that it vanishes exactly on the unipotent-torsion points.

\subsection{Acknowledgement} Both authors thank Gabriel Dill for pointing out a relevant mathoverflow post that helped to motivate this work. 

\section{Preliminaries}

\subsection{Some structure results about algebraic groups}
\label{Subsect.Structure_results}

\begin{notation}
	Throughout this subsection $G$ denotes a connected algebraic group.
	Moreover, $G_{\aff}$ denotes the largest connected affine algebraic subgroup of $G$ and
	$G_{\ant}$ denotes the smallest normal algebraic subgroup of $G$ such that $G/G_{\ant}$ is affine.
\end{notation}

By \cite[Theorem~16, p. 439]{Ro1956Some-basic-theorem}, $G_{\aff}$ is the unique closed normal 
connected affine subgroup of $G$ such that $G/G_{\aff}$ is abelian.
The subgroup $G_{\ant}$ is connected, anti-affine (i.e.~$\kk[G_{\ant}] = \kk$) and central in $G$
(see \cite[Ch.~III, \S3, Th\'eor\`eme~8.2, Corollaire~8.3]{DeGa1970Groupes-algebrique}), and 
$G$ can be written as $G = G_{\ant} G_{\aff}$ (so called ``Rosenlicht decomposition'', 
see \cite[Corollary~5, p.~440]{Ro1956Some-basic-theorem}). 


Moreover, we will constantly use the fact that for a homomorphism $f \colon G \to H$ of algebraic groups the image
$f(G)$ is closed in $H$, see e.g.~\cite[Exp.~IV$_B$, Proposition~1.2]{2011Schemas-en-groupes}.

\begin{lemma}
	\label{Lem.Homomorphic_images}
	Let $f \colon G \to H$ be a surjective 
	homomorphism of connected algebraic groups. Then $f(G_{\ant})  = H_{\ant}$
	and $f(G_{\aff}) = H_{\aff}$.
\end{lemma}

\begin{proof}
	Since $\kk[G_{\ant}] = \kk$, we get $\kk[f(G_{\ant})] = \kk$. In particular, $f(G_{\ant})$
	maps to the neutral element via $H \to H / H_{\ant}$, as $H / H_{\ant}$ is affine. Hence
	$f(G_{\ant}) \subseteq H_{\ant}$. On the other hand, $f$ descents to a surjection
	$G/G_{\ant} \to H/f(G_{\ant})$ and since $G/G_{\ant}$ is affine, it follows that $H/f(G_{\ant})$ is affine
	(see~e.g.~\cite[Corollary~2, p. 440]{Ro1956Some-basic-theorem}). This shows that $H_{\ant} \subseteq f(G_{\ant})$.
	
	Note that $f(G_{\aff})$ is a connected affine subgroup of $H$.
	As $f$ is surjective, $f(G_{\aff})$ is normal in $H$ and $f$ descents to a surjection
	$G/G_{\aff} \to H/f(G_{\aff})$. Since $G/G_{\aff}$ is abelian, it follows that $H/f(G_{\aff})$ is abelian as well.
	Hence, $f(G_{\aff}) = H_{\aff}$.
\end{proof}

\begin{remark}
	\label{Rem.Homomorphic_images}
	The proof shows: If $f \colon G \to H$ is only a homomorphism of connected algebraic groups, then
	we have still $f(G_{\ant}) \subseteq H_{\ant}$ and $f(G_{\aff})\subseteq H_{\aff}$.
\end{remark}

\begin{lemma}
	\label{Lem.Unipot_Radical} 
	The unipotent radical $R_u(G_{\aff})$ of $G_{\aff}$ is the largest  
	closed normal unipotent subgroup of $G$.
\end{lemma}

\begin{proof}
	Let $N \subseteq G$ be a closed normal unipotent subgroup of $G$. 
	As $N$ is solvable, it is affine (see \cite[Corollary~4, p. 440]{Ro1956Some-basic-theorem}).
	Hence $N$ is contained in $G_{\aff}$ and thus $N \subseteq R_u(G_{\aff})$.
\end{proof}

\begin{lemma}
	\label{Lem.Aff_Ant_of_GmodR_uG}
	Let $\hat{G} \coloneqq G / R_u(G_{\aff})$. Then $\hat{G}_{\ant}$ is a semi-abelian variety (i.e.~an algebraic group 
	extension of an abelian variety by some torus) and $\hat{G}_{\aff}$ is an affine reductive group (i.e.~$R_u(\hat{G}_{\aff})$ is trivial).
\end{lemma}	
	 
\begin{proof}
	By Lemma~\ref{Lem.Homomorphic_images}, the canonical projection $p \colon G \to \hat{G}$ restricts
	to a surjection $q \colon (G_{\ant})_{\aff} \to (\hat{G}_{\ant})_{\aff}$ with $\ker(q) = R_u(G_{\aff}) \cap (G_{\ant})_{\aff}$.
	As $R_u((G_{\ant})_{\aff})$ is the largest unipotent subgroup of $G_{\ant}$
	(by Lemma~\ref{Lem.Unipot_Radical} applied to the commutative group $G_{\ant}$), we get
	$R_u((G_{\ant})_{\aff}) = \ker(q)$. Thus $(\hat{G}_{\ant})_{\aff} \simeq  (G_{\ant})_{\aff}/\ker(q)$ is  a torus.
	Now, $\hat{G}_{\ant}$ is semi-abelian, as an extension of the abelian variety $\hat{G}_{\ant}/(\hat{G}_{\ant})_{\aff}$ by the torus $(\hat{G}_{\ant})_{\aff}$.
	
	Again by Lemma~\ref{Lem.Homomorphic_images}, $p \colon G \to \hat{G}$ restricts to a surjection  $G_{\aff} \to \hat{G}_{\aff}$
	with kernel $R_u(G_{\aff})$. Hence $\hat{G}_{\aff}$ is affine and reductive.
\end{proof}	 
	 
For the rest of this subsection, we focus on the case, where $G_{\ant}$ is semi-abelian.

\begin{definition}
	Assume that $G_{\ant}$ is semi-abelian.
	We call an element $g \in G$ \emph{semisimple} if $g = g_1g_2$, where 
	$g_1 \in G_{\ant}$ and $g_2  \in G_{\aff}$ is semisimple (in the usual sense).
\end{definition}

If $f \colon G \to H$ is a homomorphism of connected algebraic groups where $G_{\ant}$ and $H_{\ant}$ are semi-abelian,
then for every semisimple $g \in G$ the image $f(g)$ is semisimple in $H$, since
$f(G_{\ant}) \subseteq H_{\ant}$ and $f(G_{\aff}) \subseteq H_{\aff}$ 
(see Remark~\ref{Rem.Homomorphic_images}).

Moreover, we have the following description of the torsion points:

\begin{lemma}
	\label{Lem.torsion_points}
	Assume that $G_{\ant}$ is semi-abelian and fix a maximal torus $T \subseteq G_{\aff}$.
	Then, the conjugates of the
	\begin{enumerate}[leftmargin=*, label=\alph*)]
	\item \label{Lem.torsion_points1}  elements in $G_{\ant}T$ are exactly the semisimple elements in $G$.
	\item \label{Lem.torsion_points2}  torsion points in $G_{\ant}T$ are exactly the torsion points in  $G$.
	\end{enumerate}
\end{lemma}

\begin{proof}
	\ref{Lem.torsion_points1}: This follows from the fact that the semisimple elements of $G_{\aff}$
	are exactly the conjugates of the elements in $T$, see \cite[Theorem~22.2]{Hu1975Linear-algebraic-g}.
	
	\ref{Lem.torsion_points2}:
	Let $g \in G$ be a torsion point and write $g = g_1 g_2$, where $g_1 \in G_{\ant}$ and $g_2 \in G_{\aff}$.
	Then there exists $n > 0$ such that $1 = g^n = g_1^n g_2^n$. By possibly enlarging $n$,
	we may assume that $g_1^{-n} = g_2^n \in (G_{\ant} \cap G_{\aff})^\circ \eqqcolon F$.
	Since $F$ is a torus, it follows that $g_2^n$ is semisimple. Using the classical Jordan decomposition of 
	$g_2$ in $G_{\aff}$ yields that $g_2$ is semisimple. Hence, there exists $a \in G_{\aff}$ with $ag_2 a^{-1} \in T$
	and thus $g$ is conjugated to $g_1  ag_2 a^{-1} \in G_{\ant} T$. 
\end{proof}

We finish this subsection with a description of the Jordan decomposition in case $G_{\ant}$ is semi-abelian.
Note that we don't have such a Jordan decomposition in general, even not in case when $G$ is commutative.

\begin{lemmadef}[Jordan decomposition]
	\label{Lem.Jordan}
	Assume that $G_{\ant}$ is semi-abelian.
	Let $g \in G$. Then there exist unique $g_u, g_s \in G$ such that  $g = g_u g_s$, $g_u$ is unipotent, 
	$g_s$ is semisimple and  $g_u, g_s$ commute. 
	Moreover, if $g$ commutes with some $b \in G$, then $g_s$ and $g_u$ commute with $b$ as well.
	
	The elements $g_u$ and $g_s$ are called \emph{unipotent part} and \emph{semisimple part} of $g$, respectively.
\end{lemmadef}

\begin{proof}
	We will constantly use that $G_{\ant}$ is central in $G$.
	
	Assume that $g = u s$, where $u$ and $s$ are unipotent and semisimple, respectively, and $us = su$. 
	Let $b \in G$ such that $b$ commutes with $g$. Write $b = b_1b_2$ and $s = s_1 s_2$, 
	where $b_1, s_1 \in G_{\ant}$, $b_2, s_2 \in G_{\aff}$ and $s_2$ is semisimple. Then $u s_2 = s_2 u$
	and $b$ commutes with $u s_2$ and hence $b_2$ commutes with $u s_2$. Since $b_2, u, s_2 \in G_{\aff}$,
	it follows that $b_2$ commutes with $u$ and with $s_2$
	(see e.g.~\cite[Lemma~15.1B]{Hu1975Linear-algebraic-g}). Hence, $b$ commutes with $u$ and $s$.
	
	Assume that $g = us = u' s'$, where $u, s$ commute, $u', s'$ commute, $u, u'$ are unipotent and $s, s'$ are semisimple.
	Since $u'$ commutes with $g$, it follows by the previous paragraph that $u'$ commutes with $u$. Analogously,
	$s'$ commutes with $s$. Hence $(u')^{-1} u$ is unipotent and $s' s^{-1}$ is semisimple, i.e.~$s' s^{-1} \in G_{\ant} T$
	for some maximal torus $T \subseteq G_{\aff}$. Since $(u')^{-1} u = s' s^{-1} \in G_{\ant}T$ and since $G_{\ant}T$ is semi-abelian, 
	we get that $u' = u$ and thus
	$s' = s$. This shows the uniqueness of the decomposition.
	
	Write $g = g_1 g_2$, where $g_1 \in G_{\ant}$ and $g_2 \in G_{\aff}$. Let $(g_2)_u$ and $(g_2)_s$ 
	be the unipotent and semisimple part of $g_2$ in $G_{\aff}$,
	respectively, see~\cite[Lemma~15.1B]{Hu1975Linear-algebraic-g}. Then $g = (g_2)_u g_1 (g_2)_s$, $(g_2)_u$ is unipotent, $g_1 (g_2)_s$ is semisimple and 
	$(g_2)_u, g_1 (g_2)_s$ commute.
	This proves the existence of the decomposition.
\end{proof}

\subsection{Conjugation in algebraic groups with trivial unipotent radical}	 
\label{Subsect.Conjugation}
	%

\begin{notation}
	Throughout this subsection $G$ denotes a connected algebraic group such that the unipotent radical
	$R_u(G_{\aff})$ is trivial. Let $Z \subseteq G$ be the connected component of the center of $G$,
	denote by $G' \coloneqq [G, G] = [G_{\aff}, G_{\aff}] \subseteq G_{\aff}$, the commutator subgroup of $G$ and let
	$\mu \colon Z \times G'  \to G$ be the homomorphism induced by multiplication. 
	We fix a maximal torus $T'$ in $G'$
	and denote by $W' = N_{G'}(T')/T'$ the Weyl group of $G'$ with respect to $T'$, 
	where $N_{G'}(T')$ denotes the normalizer of $T'$ in $G'$. Likewise, we fix a maximal torus $T \subseteq G_{\aff}$  that contains $T'$
	and denote by $W$ the Weyl group of $G$ with respect to $T$.
\end{notation}

Recall that $G_{\ant}$ is a connected semi-abelian group and $G_{\aff}$ is a connected affine reductive group
(see Lemma~\ref{Lem.Aff_Ant_of_GmodR_uG}). Note that $G_{\ant} \subseteq Z$ and also the connected center of $G_{\aff}$
is contained in $Z$. Moreover, $G'$ is either trivial or a semi-simple algebraic group, $\mu \colon Z \times G' \to G$
is surjective and $Z \cap G'$ is finite, see~\cite[Proposition~14.2]{Bo1991Linear-algebraic-g}.

\smallskip

The following proposition gives a strong connection between the quotient by conjugation of $G$ and $G'$.
	 
\begin{prop}
	\label{Prop.Alg_quotient_general}
	The categorial quotient $\pi_{G} \colon G \to G \aquot G$ of  $G$ by conjugation exists.
	Moreover, there exists a (unique) surjective mor\-phism 
	$\kappa \colon Z \times (G' \aquot G') \to G \aquot G$ such  that
	\begin{equation}
		\label{Eq.Alg_quotient_general}
		\begin{gathered}
		\xymatrix@=20pt{
			Z \times G' \ar[d]_-{\id_{Z} \times \pi_{G'}} \ar[r]^-{\mu} &  G
			\ar[d]^-{\pi_{G}}   \\
			Z \times (G' \aquot G') \ar[r]^-{\kappa} & G \aquot G 
		}				
		\end{gathered}
	\end{equation}
	commutes and is cartesian, where $\pi_{G'} \colon G' \to G' \aquot G'$
	denotes the morphism induced by $\kk[G']^{G'} \subseteq \kk[G']$
	and $\kk[G']^{G'}$ is the $\kk$-subalgebra of all function in $\kk[G']$ that are invariant under $G'$-conjugation 
	($\pi_{G'}$ is in fact the categorical quotient, see~\cite[Ch.~13, Theorem~2.4 and Theorem~2.12]{SaRi2017Actions-and-invari}).
\end{prop}

\begin{remark}
	\label{Rem.Alg_quotient_general}
	The proof will show that there is an action of $S \coloneqq Z \cap G'$ on
	$Z \times (G' \aquot G')$ such that $\kappa$ is a principal $S$-bundle with respect 
	to this $S$-action and $\id_{Z} \times \pi_{G'}$ is $S$-equivariant
	with respect to this action and the following $S$-action on $Z \times G'$:
	\[
	s \cdot (z, g) \coloneqq (z s, s^{-1} g) \quad 
	\textrm{for all $s \in S$,  $(z, g) \in Z \times G'$} \, .
	\]
\end{remark}

\begin{proof}[Proof of Proposition~\ref{Prop.Alg_quotient_general}]
	Let $S \coloneqq Z \cap G' \subseteq G$. Then $\mu$ is a principal $S$-bundle 
	by the $S$-action $s \cdot (z, g) \coloneqq (z s, s^{-1} g)$ on $Z \times G'$.
	Note that
	\[
		\pi_{G'}(s h g h^{-1}) = \pi_{G'}(h s g h^{-1}) = \pi_{G'}(s g)
	\]
	for all $s \in S$,  $h,  g \in G'$.
	Hence there exists an $S$-action $\varepsilon$ on $G' \aquot G'$ 
	such that 
	\[
		\pi_{G'}(s g) = \varepsilon(s, \pi_{G'}(g)) \quad \textrm{for all $s \in S$, $g \in G'$}
	\] 
	and thus $\pi_{G'}$ is $S$-equivariant. 
	By \cite[Proposition~4]{Se1958Espaces-fibres-alg}, there exists
	the geometric quotient
	\[
		\kappa \colon Z \times (G' \aquot G') \to Z \times^S (G' \aquot G') 
		\eqqcolon G \aquot G
	\]
	by the $S$-action $s \cdot (z, y) \coloneqq (z s, \varepsilon(s^{-1}, y))$ 
	on $Z \times (G' \aquot G')$
	and in fact $\kappa$ is a principal $S$-bundle. Moreover,  
	\[
		q \coloneqq \id_{Z} \times \pi_{G'}
	\]
	is $S$-equivariant with respect to the above constructed $S$-actions.
	Thus $\kappa \circ q$ is $S$-invariant and hence there exists a unique
	morphism $\pi_{G} \colon G \to G \aquot G$ such that~\eqref{Eq.Alg_quotient_general} commutes.
	Moreover, since $\mu$ and $\kappa$ are principal $S$-bundles and since $q$ is $S$-equivariant,
	it follows that~\eqref{Eq.Alg_quotient_general} is cartesian, see \cite[\S3.1]{Se1958Espaces-fibres-alg}.
	
	Now, we show that $\pi_G \colon G \to G \aquot G$ is the categorial quotient of $G$ by conjugation.
	Since $q$ is invariant under conjugation, it follows by the commutativity of~\eqref{Eq.Alg_quotient_general} that
	$\pi_G \circ \mu$ is invariant under conjugation. Since $\mu$ is a surjective group homomorphism, 
	$\pi_G$ is invariant under conjugation.
	
	Let $\varphi \colon G \to V$ be a morphism that is invariant under conjugation. Then 
	$\varphi \circ \mu$ is invariant under conjugation, since $\mu$ is a group homomorphism.
	In particular, there exists a morphism 
	\[
		\psi \colon Z \times (G' \aquot G') \to W
	\] 
	such that $\psi \circ q = \varphi \circ \mu$. Since $q$ is surjective and $S$-equivariant and since
	$\mu$ is $S$-invariant, it follows that $\psi$ is $S$-invariant. Hence, there exists a morphism
	$\theta \colon G \aquot G \to Z$ such that $\psi = \theta \circ \kappa$. This implies that
	\[
		\varphi \circ \mu = \psi \circ q =  \theta \circ \kappa \circ q = \theta \circ \pi_{G} \circ \mu \, .
	\]
	By the surjectivity of $\mu$ we get thus $\varphi = \theta \circ \pi_{G}$. This shows that
	$\pi_{G}$ is a categorical quotient of $G$ by conjugation.
\end{proof}

We denote for a group $H$ and $h \in H$ the conjugation class of $h$ in $H$ by $C_H(h)$.
The following lemma is evident:

\begin{lemma}
	\label{Lem.Conj_class_general}
	Let $(z, g) \in Z \times G'$. Then we have
	\[
			C_G(zg) = z C_{G'}(g) = \mu(\{z\} \times C_{G'}(g))   \, .	
	\]
\end{lemma}

\begin{corollary}
	\label{Cor.Properties_Alg_quotient_general}
	The categorial quotient $\pi_G \colon G \to G \aquot G$
	is surjective, $G \aquot G$ is irreducible and normal, and every fibre of $\pi_G$ contains
	exactly one closed conjugacy class.
\end{corollary}

\begin{proof}
	 Since $G' \aquot G'$ is normal and irreducible (see e.g.~\cite[Ch.~6, Theorem 4.8]{SaRi2017Actions-and-invari}) and since
	$\kappa$ in~\eqref{Eq.Alg_quotient_general} of Proposition~\ref{Prop.Alg_quotient_general} (see also Remark~\ref{Rem.Alg_quotient_general}) 
	is a principal bundle (in particular $\kappa$ is faithfully flat), it follows that $G \aquot G$ is normal and irreducible, see e.g.~\cite[Proposition~14.57]{GoWe2010Algebraic-geometry}.
	Since every fibre of $\pi_{G'} \colon G' \to G' \aquot G'$ contains exactly
	one closed conjugacy class (see e.g.~\cite[Ch.~13, Theorem 2.4]{SaRi2017Actions-and-invari}), it follows from Proposition~\ref{Prop.Alg_quotient_general}
	in combination with Lem\-ma~\ref{Lem.Conj_class_general} that every fibre of $\pi_G$ contains exactly one closed
	conjugacy class.
\end{proof}

For the proofs of the next two lemmas we will constantly use the following fact: 
If an element of a reductive group commutes with a maximal torus, then it lies in that maximal torus,
see e.g.~\cite[Corollar~26.2A]{Hu1975Linear-algebraic-g}.

\begin{lemma}
	
	\label{Lem.GantT=ZT'}
	We have $G_{\ant} T = Z T'$ inside $G$.
\end{lemma}

\begin{proof}
	Note that $G_{\ant} \subseteq Z$ and by assumption $T' \subseteq T$. 
	
	Let $t' \in T'$. Then there exist
	$a \in G_{\ant}$ and $g \in G_{\aff}$ with $t' = a g$. As $t'$ commutes with all elements of $T$, $g$ commutes with all elements of $T$
	and thus $g \in T$. This implies $t' \in G_{\ant}T$.
	
	On the other hand, let $t \in T$. Then there exist $z \in Z$ and $h \in G'$ with $t = zh$. As $t$ commutes with $T'$, it follows that
	$h$ commutes with $T'$ and thus $h \in T'$. This gives $t \in Z T'$.
\end{proof}

\begin{remark}
	\label{Rem.Jordan}
	Let $g \in G$ and write $g = zg'$ for $z \in Z$ and $g' \in G'$. Then $z (g')_s$ and $(g')_u$ commute, 
	$z (g')_s$ is semisimple (by Lemma~\ref{Lem.torsion_points} and Lemma~\ref{Lem.GantT=ZT'}) and $(g')_u$ is unipotent. Hence
	\[
	g_s = z (g')_s \, .
	\]
\end{remark}

\begin{lemma}
	\label{Lem.GemQuot_by_W_W'}
	The inclusion $N_{G'}(T') \subseteq N_{G_{\aff}}(T)$ induces a group isomorphism $\gamma \colon W' \to W$ and for all $x \in G_{\ant} T = ZT'$ and all $w' \in W'$
	we have
	\begin{equation}
		\label{Eq.Conjugation}
		w' x (w')^{-1} = \gamma(w') x \gamma(w')^{-1} \in G_{\ant}T = ZT' \, . 
	\end{equation}
\end{lemma}

\begin{proof}
	As the intersection $G' \cap T$ commutes with $T'$ it is equal to $T'$.
	 Hence, the inclusion
	$\iota \colon N_{G'}(T') \to N_{G_{\aff}}(T)$ induces an injective group homomorphism $\gamma \colon W' \to W$.
	If $g \in N_{G_{\aff}}(T)$, the we may write $g =zg'$ with $z \in Z$ and $g' \in G'$. As $G' \cap T = T'$,
	we get $g' T' (g)^{-1} = T'$, i.e.~$g' \in N_{G'}(T')$. As $g, g' \in G_{\aff}$ we get $z \in Z \cap G_{\aff}$ and thus $z \in T$. Hence
	$\gamma(g'T') = gT$, and thus $\gamma$ is surjective.
	
	Note that $G_{\ant}T = ZT' = \mu(Z \times T')$ is a closed algebraic subgroup of $G$.
	As $G_{\ant}$ is central in $G$, $T$ commutes with $G_{\ant}T$ and thus $W$ acts by conjugation on $G_{\ant}T$
	Likewise, $W'$ acts by conjugation on $ZT'$. The equation~\eqref{Eq.Conjugation} is now evident.
\end{proof}

\begin{corollary}
	\label{Cor.Restriction_of_alg_quotient_general}
	The restriction
	\[
	\rho_{G, T} \coloneqq \pi_G |_{G_{\ant}T} \colon G_{\ant}T \to G \aquot G
	\]
	is a geometric quotient for the $W$-action by conjugation on $G_{\ant}T$.
\end{corollary}

\begin{proof}
	By Lemma~\ref{Lem.GemQuot_by_W_W'}, the geometric quotients of $G_{\ant} T = ZT'$ by conjugation with $W$ and $W'$  coincide.
	Thus we may work with conjugation by $W'$ on $ZT'$.
		
	Let $\eta \colon ZT' \to ZT' / W'$ be the geometric quotient with respect to 
	$W'$-conjugation. As $\rho_{G, T}$ is invariant under $W'$-conjugation, there exists a morphism 
	$f \colon ZT' / W' \to G \aquot G$ such that $\rho_{G, T} = f \circ \eta$. We have to show that $f$ is an isomorphism.
	
	Let $S = Z \cap G'$. Then $S$ is contained in $T'$ (as $S$ is central in $G'$, see e.g.~\cite[Corollary~26.2A]{Hu1975Linear-algebraic-g}).
	Hence, $\mu^{-1}(Z T') = Z \times T'$ and thus the cartesian 
	diagram~\eqref{Eq.Alg_quotient_general} in Proposition~\ref{Prop.Alg_quotient_general} restricts
	to the following cartesian diagram
	\begin{equation}
		\label{Eq.Alg_quotient_general_restricted}
		\begin{gathered}
			\xymatrix@=20pt{
				Z \times T' \ar[d]_-{\id_{Z} \times \rho_{G', T'}} \ar[r]^-{\mu |_{Z \times T'}} &  Z T'
				\ar[d]^-{\rho_{G, T}}   \\
				Z \times (G' \aquot G') \ar[r]^-{\kappa} & G \aquot G \, ,
			}				
		\end{gathered}
	\end{equation}
	where $\rho_{G', T'}$ is the restriction of $\pi_{G'}$ to $T'$.
	By \cite[Corollary~6.4]{St1965Regular-elements-o}, $\rho_{G', T'}$ is a geometric quotient of $T'$
	by $W'$-conjugation. As $W'$ acts faithfully on $T'$, we get that $\rho_{G', T'}$ is finite of degree $|W'|$ and surjective.
	As $\kappa$ is faithfully flat, $\rho_{G, T'}$ is finite of degree $|W'|$
	and surjective
	(see e.g. \cite[Proposition~14.51]{GoWe2010Algebraic-geometry}). Since $\eta \colon ZT' \to ZT' / W'$
	is finite of degree $|W'|$, it follows thus that $f \colon ZT' / W' \to G \aquot G$ is finite of degree one and surjective.
	Since $G \aquot G$ is normal and irreducible and since $ZT' / W'$ is irreducible (see Corollary~\ref{Cor.Properties_Alg_quotient_general}), it follows by
	Zariski's main theorem (see~\cite[Corollaire 4.4.9]{Gr1961Elements-de-geomet-III}) that $f$ is an isomorphism. 
\end{proof}

We give now further properties of the categorical quotient $\pi_G \colon G \to G \aquot G$.

\begin{corollary}
	\label{Cor.Conjugacy_classes_general} $ $
	\begin{enumerate}[leftmargin=*]
		\item \label{Cor.Conjugacy_classes_general1} For all $g \in G$, 
		the intersection $G_{\ant}T \cap \pi_G^{-1}(\pi_G(g))$ is exactly one 
		orbit under the conjugacy by the Weyl group $W$ on $G_{\ant}T$.
		\item \label{Cor.Conjugacy_classes_general2} The closed conjugacy classes in 
		$G$ are the semisimple conjugacy classes, i.e.~the conjugacy classes $C_G(s), s \in G_{\ant}T$.
		\item \label{Cor.Conjugacy_classes_general3} For all $g \in G$, the fibre
		$\pi_G^{-1}(\pi_G(g))$ contains exactly one semisimple 
		conjugacy class $C$ and $C$ is contained in $\overline{C_G(g)}$.
		Moreover, $C = C_G(g_s)$ and thus $\pi_G(g) = \pi_G(g_s)$.
		\item \label{Cor.Conjugacy_classes_general4} For $g \in G$ letting 
		$E_{g} \coloneqq \set{u \in G}{\textrm{$u \in G$ unipotent, $u g_s = g_s u$}}$ gives
		\[
		\pi_G^{-1}(\pi_G(g)) =  
		\bigcup_{u \in E_{g}} C_G(u g_s) \, .
		\]
		\item \label{Cor.Conjugacy_classes_general5} Every fibre of $\pi_G \colon G \to G \aquot G$ is irreducible
		and has dimension $\dim G_{\aff} - \rank(G_{\aff})$.
		
		\item \label{Cor.Conjugacy_classes_general6} We have 
		$G_{\ant}T \cap \pi_G^{-1}(\pi_G(X)) = G_{\ant} T \cap X$ for any closed subset $X \subseteq G$ that is 
		stable under conjugation.
	\end{enumerate}
\end{corollary}

\begin{proof}
	\eqref{Cor.Conjugacy_classes_general1}: Since $\rho_{G, T}^{-1}(\rho_{G, T}(g)) = G_{\ant} T \cap \pi_G^{-1}(\pi_G(g))$,
	the result follows from Corollary~\ref{Cor.Restriction_of_alg_quotient_general}.
	
	\eqref{Cor.Conjugacy_classes_general2}:  By \cite[Proposition~18.2]{Hu1975Linear-algebraic-g} the conjugacy classes of semisimple elements in $G'$ are closed. 
	As every semisimple elemet in $G$ is conjugated to an element in $G_{\ant}T = ZT'$ (see Lemma~\ref{Lem.torsion_points}) it follows by
	using Lemma~\ref{Lem.Conj_class_general} that the conjugacy class of every semisimple element of $G$ is closed.
	On the other hand, let $C_{G}(g)$ be a closed conjugacy class in $G$. From~\eqref{Cor.Conjugacy_classes_general1} it follows that there exist $z \in Z$ and 
	$t' \in T'$ with $zt' \in \pi_G^{-1}(\pi_G(g))$.
	Hence, $C_G(g) = C_G(zt') = z C_{G'}(t')$ as they are closed conjugacy classes in $\pi_G^{-1}(\pi_G(g))$,
	see Corollary~\ref{Cor.Properties_Alg_quotient_general}.
	
	\eqref{Cor.Conjugacy_classes_general3}: Corollary~\ref{Cor.Properties_Alg_quotient_general} 
	and~\eqref{Cor.Conjugacy_classes_general2} imply that $\pi_G^{-1}(\pi_G(g))$ contains exactly
	one semisimple conjugacy class $C$. As the closure of $C_G(g)$ contains a closed
	conjugacy class and is contained in $\pi_G^{-1}(\pi_G(g))$, it follows that $C$ is contained in 
	$\overline{C_G(g)}$. Write $g = z g'$ with $z \in Z$ and $g' \in G'$. By Remark~\ref{Rem.Jordan}
	we get $g_s = z (g')_s$. By \cite[Proposition~1.7]{Hu1995Conjugacy-classes-} we have
	$(g')_s \in \overline{C_{G'}(g')}$ and thus
	\[
		g_s = z (g')_s \in z \overline{C_{G'}(g')} = \overline{C_G(z g')} = \overline{C_G(g)} \, .
	\]
	Hence $C = C_{G}(g_s)$.
	
	\eqref{Cor.Conjugacy_classes_general4}: 
	For $g \in G$, let
	\[
		S_g \coloneqq \bigcup_{u \in E_{g}} C_G(u g_s)  \, .
	\]
	
	``$\pi_G^{-1}(\pi_G(g)) \subseteq S_g$'': If $h \in \pi_G^{-1}(\pi_G(g))$, then $\pi_G(h_s) = \pi_G(h) = \pi_G(g)
	= \pi_G(g_s)$ and there exists $a \in G$ with  $h_s = a g_s a^{-1}$ by~\eqref{Cor.Conjugacy_classes_general3}
	Hence, $u \coloneqq a^{-1} h_u a \in G$ is unipotent and since
	\[
	a u g_s a^{-1} =  h_u ag_s a^{-1} =  h = ag_s a^{-1} h_u = ag_s u a^{-1}
	\]
	we get $u \in E_{g}$ and $h \in C_G(u g_s)$.
	
	``$\pi_G^{-1}(\pi_G(g)) \supseteq S_g$'': For $u \in E_{g}$ we get $\pi_G(u g_s) = \pi_G(g_s) = \pi_G(g)$,
	i.e.~$u g_s \in \pi_G^{-1}(\pi_G(g))$.
	
	\eqref{Cor.Conjugacy_classes_general5}: 
	By~\cite[\S3.8, Theorem~1]{St1974Conjugacy-classes-} all fibres of $\pi_{G'}$
	are irreducible and of dimension $\dim G' - \rank(G') = \dim G_{\aff} - \rank(G_{\aff})$.
	Hence, the statement follows from the cartesian diagram~\eqref{Eq.Alg_quotient_general} in Proposition~\ref{Prop.Alg_quotient_general}.
	
	\eqref{Cor.Conjugacy_classes_general6}: If $h \in G_{\ant}T \cap \pi_G^{-1}(\pi_G(x))$ for some $x \in X$, then
$h \in C_G(h) \subseteq \overline{C_G(x)} \subseteq X$ by~\eqref{Cor.Conjugacy_classes_general3}.
\end{proof}

\begin{lemma}
	\label{Lem.G'-Closedness}
	Consider
	\[
	q \coloneqq \id_Z \times \pi_{G'} \colon Z \times G' \to Z \times (G' \aquot G') \, .
	\]
	Then for every closed subvariety $Y \subseteq Z \times G'$ that is stable under $G'$-conjugation on the second factor,
	the image $q(Y)$ is closed in $Z \times (G' \aquot G')$.
\end{lemma}

\begin{proof}
	Let $Z_1, \ldots, Z_m \subseteq Z$ be open affine subsets that cover $Z$ and let
	\[
	q_i \colon q^{-1}(Z_i \times (G' \aquot G')) = Z_i \times G' \xrightarrow{(z_i, g) \mapsto q(z_i, g)} Z_i \times (G' \aquot G')
	\]
	Then $q_i$ is the categorical quotient of $Z_i \times G'$ with respect to $G'$-conjugation on the second factor.
	Since $Y \cap (Z_i \times G')$ is closed in the affine variety $Z_i \times G'$ and stable under $G'$-conjugation, it follows
	that $q_i(Y \cap (Z_i \times G')) = q(Y) \cap (Z_i \times (G' \aquot G'))$ is closed in $Z_i \times (G' \aquot G')$, see e.g.~\cite[Ch. 13, Theorem~2.4]{SaRi2017Actions-and-invari} for $i=1, \ldots, m$. 
	This implies that $q(Y)$ is closed in $Z \times (G' \aquot G')$.
\end{proof}

\begin{corollary}
	\label{Cor.Preimage_irred_under_alg_quotient_general}
	For every irreducible closed subvariety $Z \subseteq G \aquot G$, the
	preimage $\pi_{G}^{-1}(Z)$ is irreducible.
\end{corollary}

\begin{proof}
	Let $d \coloneqq \dim G_{\aff} - \rank G_{\aff}$. By Corollary~\ref{Cor.Conjugacy_classes_general}\eqref{Cor.Conjugacy_classes_general5} every fibre of $\pi_{G}$ is irreducible
	and has dimension $d$.
	Let $X$ be an irreducible component of $\pi_{G}^{-1}(Z)$ with
	$\dim X = \dim \pi_{G}^{-1}(Z)$. We have to show that $X = \pi_{G}^{-1}(Z)$.
	
	By construction, $\pi_{G} |_X \colon X \to Z$ is dominant and every fibre
	over some open dense subset of $Z$ has dimension $d$.
	Since all fibres of $\pi_{G}$ have dimension $d$, it follows
	by the semi-continuity of the fibre dimension (see e.g.~\cite[Theorem 14.110]{GoWe2010Algebraic-geometry}) 
	that every non-empty fibre of  
	$\pi_{G} |_X \colon X \to Z$ has dimension $d$.
	Since all fibres of $\pi_G$ are irreducible, there exists thus
	a subset $V \subseteq Z$ such that $X = \pi_{G}^{-1}(V)$. 
	In particular $X$ is stable under conjugation.
	
	Let $Y \coloneqq \mu^{-1}(X) \subseteq Z \times G'$. Then $Y$ is invariant under
	conjugation and thus $q(Y)$ is closed in $Z \times (G' \aquot G')$, where 
	$q = \id_{Z} \times \pi_{G'}$, 
	see Lemma~\ref{Lem.G'-Closedness}. Let $S \coloneqq Z \cap G'$ act
	on $Z \times G'$ via $s \cdot (z, g) \coloneqq (zs, s^{-1} g)$.
	Then $Y$ is $S$-stable. By Remark~\ref{Rem.Alg_quotient_general} there exists an $S$-action
	on $Z \times (G' \aquot G')$ such that 
	$q$ is $S$-equivariant and the surjection $\kappa$ in the diagram~\eqref{Eq.Alg_quotient_general} from
	Proposition~\ref{Prop.Alg_quotient_general} is a principal $S$-bundle.
	Hence, $q(Y)$ is a closed $S$-stable subset of $Z \times (G' \aquot G')$
	and therefore $\kappa(q(Y))$ is closed in $G \aquot G$. Since $\kappa \circ q = \pi_{G} \circ \mu$,
	it follows that $\pi_{G}(\mu(Y)) = \pi_{G}(X)$ is closed in $G \aquot G$.
	As $\pi_{G}(X)$ is dense in $Z$, we get $V = \pi_{G}(X) = Z$, i.e.~$X = \pi_{G}^{-1}(V) = \pi_{G}^{-1}(Z)$. 
\end{proof}

\section{Commutative algebraic groups with dense torsion subgroup}

This section is devoted to a characterization of commutative algebraic groups
with dense torsion subgroup $G_{\tor}$. 

\begin{prop}\label{Prop.dense}  
	Let $G$ be a commutative algebraic group. Then 
	$G_{\tor}$ is dense in $G$ if and only if 
	every  homomorphism $\varphi \colon G \to \GG_a$ is constant. 
\end{prop}

For the proof we need some preparation.

\begin{lemma}
	\label{Lem.Comm_conn_group_div}
	Every commutative connected algebraic group is divisible.
\end{lemma}

\begin{proof}
	Let $H$ be a commutative connected algebraic group. Then
	$H_{\aff}$ is a connected commutative affine algebraic group and hence it is divisible.
	Moreover, as the quotient $H/H_{\aff}$ is an abelian variety, it is divisible, and thus
	$H$ is divisible as an extension of commutative divisible groups.
\end{proof}

\begin{lemma}
	\label{Lem.Conn_Comp_Comm}
	If $G$ is a commutative algebraic group, then $G = \bigcup_{i=1}^n t_i G^\circ$ where
	$t_1, \ldots, t_n \in G$ are torsion points  and $G^\circ$ denotes the connected component of $G$.
\end{lemma}

\begin{proof}
	Let $n$ be the index of $G^\circ$ in $G$. Then there exist $g_1, \ldots g_n \in G$
	such that $G$ is the union of the subsets $g_1 G^\circ, \ldots, g_n G^\circ$. 
	By Lemma~\ref{Lem.Comm_conn_group_div}, $G^\circ$ is divisible. As
	$g_i^n \in G^\circ$, there exists thus $a_i \in G^\circ$ such that $g_i^n = a_i^n$.
	Then the torsion elements $t_i = g_i a_i^{-1} \in G$, $i=1, \ldots n$ will do the job.
\end{proof}

%
%
%

\begin{proof}[Proof of Proposition~\ref{Prop.dense}]  
	If $G_{\tor}$ is dense in $G$, then every homomorphism $G \to \GG_a$ is constant. 
	Now,  assume that every homomorphism $G \to \GG_a$ is constant.

	Let $G^\circ$ be the connected component of $G$. If $(G^\circ)_{\tor}$ is dense in $G^\circ$, then
	$G_{\tor}$ is dense in $G$ by Lemma~\ref{Lem.Conn_Comp_Comm}. Thus we may and will assume that
	$G$ is commutative and connected.  
	
	If $\dim G = 0$, then the lemma is trivial. Hence we assume $\dim G > 0$.
	If $G$ has only finitely many torsion points, then $G$ is affine (as $G_{\aff}$ is divisible and $(G/G_{\aff})_{\tor}$
	lies dense in $G/G_{\aff}$, see
	Lemma~\ref{Lem.Comm_conn_group_div}) and thus $G \simeq \GG_a^r$ for some $r > 0$.
	As every homomorphism $G \to \GG_a$ is constant, we conclude that $G$ contains infinitely many
	torsion points. 
	Moreover, as  $G$ is commutative, the torsion points $G_{tor}$ from a normal subgroup of $G$. As $G$ contains infinitely many torsion points, the connected component of the closure   
	$H \coloneqq \overline{G_{\tor}}^\circ \subseteq G$ is a normal connected subgroup of positive di\-men\-sion, which
	is divisible (see Lemma~\ref{Lem.Comm_conn_group_div}). 
	This implies that the canonical projection
	$G \to G/H$ restricts to a surjection $G_{\tor} \to (G/H)_{\tor}$.
	As $\dim G/H < \dim G$, we  conclude the proposition by induction. 
\end{proof} 

The following example shows that already in a solvable (non-commutative) algebraic group the set
of torsion points can be rather complicate to describe:

\begin{example}
	\label{Exa.Torsion_in_Borel}
	Let $B \subseteq \GL_3(\kk)$ be the subgroup of upper triangular matrices. Then we may and will identify
	$B$ with $H \rtimes T$, where $T \subseteq B$ denotes the subgroup of diagonal matrices,
	$H \subseteq B$ denotes the subgroup of upper triangular matrices with only $1$ on the diagonal
	and the product is given by $(A, S) \cdot (A', S') = (ASA'S^{-1}, SS')$.
	We will compute the set of torsion points $B_{\tor}$. For this we insert some notation:
	\begin{eqnarray*}
		S_0 &\coloneqq& \Bigset{\begin{psmallmatrix}
			\lambda & 0 & 0 \\ 0 & \mu & 0 \\ 0 & 0 & \varepsilon
		\end{psmallmatrix} \in T}{\lambda = \mu = \varepsilon, 
		\, \textrm{$\lambda, \mu, \varepsilon$ are roots of unity in $\kk$}} \\
		S_1 &\coloneqq& \Bigset{\begin{psmallmatrix}
					\lambda & 0 & 0 \\ 0 & \mu & 0 \\ 0 & 0 & \varepsilon
			\end{psmallmatrix} \in T}{\lambda \neq \mu = \varepsilon, \, \textrm{$\lambda, \mu, \varepsilon$ are roots of unity in $\kk$}} \\
		S_2 &\coloneqq& \Bigset{\begin{psmallmatrix}
		\lambda & 0 & 0 \\ 0 & \mu & 0 \\ 0 & 0 & \varepsilon
	\end{psmallmatrix} \in T}{\varepsilon \neq \lambda = \mu, \, \textrm{$\lambda, \mu, \varepsilon$ are roots of unity in $\kk$}} \\
		S_3 &\coloneqq& \Bigset{\begin{psmallmatrix}
		\lambda & 0 & 0 \\ 0 & \mu & 0 \\ 0 & 0 & \varepsilon
		\end{psmallmatrix} \in T}{\mu \neq \varepsilon = \lambda, \, \textrm{$\lambda, \mu, \varepsilon$ are roots of unity in $\kk$}} \\
		S_4 &\coloneqq& \Bigset{\begin{psmallmatrix}
		\lambda & 0 & 0 \\ 0 & \mu & 0 \\ 0 & 0 & \varepsilon
		\end{psmallmatrix} \in T}{\lambda \neq \mu \neq \varepsilon \neq \lambda, 
		\, \textrm{$\lambda, \mu, \varepsilon$ are roots of unity in $\kk$}} \, .
	\end{eqnarray*}
	We claim that there is a decomposition into disjoint subsets
	\[
		B_{\tor} = \{ I_3\} \times S_0 \cup V_{H \times S_1}(c) \cup V_{H \times S_2}(a)
		\cup V_{H \times S_3}((1-\lambda \mu^{-1})b -ac) 
		\cup H \times S_4 \, ,
	\]
	where for $f \in \kk[a, b, c, \lambda^{\pm1}, \mu^{\pm1}, \varepsilon^{\pm1}]$ and $i \in \{1, 2, 3\}$ we denote
	\[
		V_{H \times S_i}(f) \coloneqq \Bigset{
		\left(
		\begin{psmallmatrix}
				1 & a & b \\ 0 & 1 & c \\ 0 & 0 & 1
		\end{psmallmatrix},
		\begin{psmallmatrix}
			\lambda & 0 & 0 \\ 0 & \mu & 0 \\ 0 & 0 & \varepsilon
		\end{psmallmatrix} \right) \in H \times S_i
		}{f(a, b, c, \lambda, \mu, \varepsilon) = 0}
	\]
	and $I_3$ denotes the identity matrix. Indeed, fix an integer $n > 0$ and let 
	\[
	P = \left( \begin{psmallmatrix}
		1 & a & b \\ 0 & 1 & c \\ 0 & 0 & 1
	\end{psmallmatrix}, 
	\begin{psmallmatrix}
		\lambda & 0  & 0  \\ 0 & \mu & 0 \\ 0 & 0 & \varepsilon
	\end{psmallmatrix} \right) \in H \rtimes T \, , \quad
		\theta \coloneqq \lambda \mu^{-1} \, , \quad \eta \coloneqq \mu \varepsilon^{-1} \, .
	\]
	Using that
	\[
		\begin{psmallmatrix}
			\lambda & 0 & 0 \\ 0 & \mu & 0 \\ 0 & 0 & \varepsilon
		\end{psmallmatrix}
		\begin{psmallmatrix}
			1 & a & b \\ 0 & 1 & c \\ 0 & 0 & 1
		\end{psmallmatrix}
		\begin{psmallmatrix}
			\lambda^{-1} & 0 & 0 \\ 0 & \mu^{-1} & 0 \\ 0 & 0 & \varepsilon^{-1}
		\end{psmallmatrix}
		= \begin{psmallmatrix}
			1 & \theta a & (\theta \eta) b \\ 0 & 1 & \eta c \\ 0 & 0 & 1
		\end{psmallmatrix}
	\]
	it follows that $P^n$ is the neutral element in $H \rtimes T$ if and only if
	\[
		\begin{psmallmatrix}
			\lambda & 0  & 0  \\ 0 & \mu & 0 \\ 0 & 0 & \varepsilon
		\end{psmallmatrix}^n = I_3
		\quad \textrm{and} \quad
		\prod_{i=0}^{n-1} \begin{psmallmatrix}
			1 & \theta^i a & (\theta \eta)^i b \\ 0 & 1 & \eta^i c \\ 0 & 0 & 1
		\end{psmallmatrix} = I_3
		 \, .
	\]
	A calculation shows that this last assertion is equivalent to  $\lambda^n = \mu^n = \varepsilon^n = 1$
	and
	\begin{equation} 
		\label{Eq.Unipotentpart}
		a \sum_{i=0}^{n-1} \theta^i = 0 \, , \
		b \sum_{i=0}^{n-1} (\theta \eta)^i + ac S(\theta, \eta) = 0 \, , \
		c \sum_{i=0}^{n-1} \eta^i = 0 \, ,
	\end{equation}
	where 
	\[
		S(\theta, \eta) \coloneqq \sum_{i=1}^{n-1} \left( \sum_{j=0}^{i-1} \theta^j \right) \eta^i \, .
	\]
	Note that in case $\theta \neq 1$ we get
	\[
		S(\theta, \eta) = \sum_{i=1}^{n-1} \frac{1- \theta^i}{1-\theta} \eta^i = \frac{1}{1-\theta} 
		\left( \sum_{i=1}^{n-1} \eta^i - \sum_{i=1}^{n-1} (\theta \eta)^i \right) \, .
	\]
	Now, we fix $\lambda, \mu, \varepsilon$ with $\lambda^n = \mu^n = \varepsilon^n = 1$ and make a case-by-case analysis:
	\begin{enumerate}[leftmargin=*, label=\arabic*)]
		\setcounter{enumi}{-1}
		\item $\theta = 1$, $\eta = 1$, $\theta\eta = 1$: \eqref{Eq.Unipotentpart} is equivalent to $a = b = c = 0$.
		\item $\theta \neq 1$, $\eta  = 1$, $\theta\eta \neq 1$: \eqref{Eq.Unipotentpart} is equivalent to $c = 0$.
		\item $\theta = 1$, $\eta  \neq 1$, $\theta\eta \neq 1$: \eqref{Eq.Unipotentpart} is equivalent to $a = 0$.
		\item $\theta \neq 1$, $\eta  \neq 1$, $\theta\eta=1$: \eqref{Eq.Unipotentpart} is equivalent to $(1- \theta) b -ac = 0$.
		\item $\theta \neq 1$, $\eta  \neq 1$, $\theta\eta \neq 1$: \eqref{Eq.Unipotentpart} gives no condition.
	\end{enumerate}
	This case-by-case analysis gives the claim, as for each $i = 0, \ldots, 4$,
	case~$i)$ corresponds to 
	\[
	\begin{psmallmatrix}
		\lambda & 0 & 0 \\ 0 & \mu & 0 \\ 0 & 0 & \varepsilon
	\end{psmallmatrix} \in S_i \, .
	\]
\end{example}

In the next example we show that Proposition~\ref{Prop.dense} fails for non-commutative groups.

\begin{example}
	Again, let $H \subseteq \GL_3(\kk)$ be the subgroup of upper triangular matrices with only $1$ on the diagonal
	and let
	\[
		T_0 \coloneqq \set{
			\begin{psmallmatrix}
				\lambda & 0 & 0 \\ 0 & \mu & 0 \\ 0 & 0 & \lambda
			\end{psmallmatrix}
		}{\lambda, \mu \in \kk^\ast} \subseteq \GL_3(\kk) \, .
	\]
	Then $T_0$ acts on $H$ by conjugation and by Example~\ref{Exa.Torsion_in_Borel} it follows that
	the closure of the torsion points in $H \rtimes T_0$ is given by
	\[
		A \coloneqq \overline{(H \rtimes T_0)_{\tor}} = \Bigset{
		\left(
		\begin{psmallmatrix}
				1 & a & b  \\ 0 & 1 & c \\ 0 & 0 & 1
		\end{psmallmatrix},
		\begin{psmallmatrix}
			\lambda & 0 & 0 \\ 0 & \mu & 0 \\ 0 & 0 & \lambda
		\end{psmallmatrix}
		\right)}
		{\begin{array}{l}
		  \textrm{$a,b, c \in \kk$, $\lambda, \mu \in \kk^\ast$,} \\
		  \textrm{$(1- \lambda \mu^{-1})b = ac$}
	  		\end{array}}	 \, .  
	\]
	Then $A$ is an irreducible hypersurface in $H \rtimes T_0$. Though $A$ is invariant under conjugation,
	but $A$ is not a subgroup of $H \rtimes T_0$, as the following example shows
	\[
			\underbrace{\left(
			\begin{psmallmatrix}
				1 & 1 & 0 \\ 0 & 1 & 0 \\ 0 & 0 & 1
			\end{psmallmatrix} \, , 
			I_3
			\right)}_{\in A}
			\underbrace{\left(
			\begin{psmallmatrix}
				1 & 0 & 0 \\ 0 & 1 & 1 \\ 0 & 0 & 1
			\end{psmallmatrix} \, , 
			I_3
			\right)}_{\in A}
			=
			\left(
			\begin{psmallmatrix}
				1 & 1 & 1 \\ 0 & 1 & 1 \\ 0 & 0 & 1
			\end{psmallmatrix} \, , 
			I_3
			\right) \not \in A \, ,
	\]
	where $I_3$ denotes the identity matrix in $\GL_3(\kk)$. Every homomorphism $\varphi \colon H \rtimes T_0 \to \GG_a$
	has $A$ in the kernel. Since $A$ is a hypersurface in $H \rtimes T_0$, but not a subgroup, 
	the kernel of $\varphi$ is the whole group $H \rtimes T_0$, i.e.~$\varphi$ is constant. 
\end{example}

\section{Manin-Mumford theorem for connected algebraic groups}


\begin{notation}
Throughout this section, let $G$ be a connected algebraic group.
For a subgroup $H \subseteq G$ and a subset $X \subseteq G$ denote 
\[
	C_G^H(X) = \set{hxh^{-1} \in G}{h \in H, x \in X} \, .
\]
If $H = G$, we simply write $C_G(X)$
instead of $C_G^G(X)$. Moreover, $Z_G(X)$ denotes the subgroup
of elements in $G$ that commute with each $x \in X$. 
\end{notation}

\begin{definition}
\label{Def.Special_subvarieties}
If $X \subseteq G$ is a subset, then we denote by $X_{\tor}$ the torsion elements from $G$ that belong to $X$.
We call the elements in
\[
	\Bigset{\overline{C_G(tS)}}{\begin{array}{l}
									\textrm{$S$ is a connected commutative subgroup of $G$} \\
									\textrm{such that $S_{\tor}$ is dense in $S$ and $t \in Z_G(S)_{\tor}$}
								\end{array}
								}
\]
\emph{special subvarieties} of $G$.
\end{definition}

In case $G$ is commutative, the special subvarieties are the closed subvarieties of the form 
$tS$, where $t \in G$ is a torsion point and $S \subseteq G$ is a commutative connected algebraic subgroup
such that $S_{\tor}$ is dense in $S$. Hence, these correspond exactly to the closed subvarieties in the conclusion of Theorem~\ref{Thm.MM}.

The aim of this section is to prove the following generalization of
the Manin-Mumford Theorem (Theorem~\ref{Thm.MM}) 
for not necessarily commutative algebraic groups:

\begin{theorem}[Manin-Mumford for algebraic groups]
	\label{Thm.MMalg}
	Let $G$ be a connected algebraic group and let $X \subseteq G$ be a subset.  
	If $X_{\tor}$ is dense in $X$,
	then $\overline{C_G(X)}$ is a finite union of special subvarieties of $G$.
\end{theorem}

In  the first section, we gather some simple topological ingredients that are needed for the proof of Theorem~\ref{Thm.MMalg}
and in the second section we provide an example that shows that the intersection of two 
special subvarieties need not be the finite union of special subvarieties. In the third section, we prove Theorem~\ref{Thm.MMalg} for certain special cases
and reduced the full proof of Theorem~\ref{Thm.MMalg} to this special cases at the end of this third section.

\subsection{Simple topological ingredients}

Here we list several simple topological lemmas that we use in the proof of Theorem~\ref{Thm.MMalg}.

\begin{lemma}
	\label{Lem.top}
	Let $f \colon X \to Y$ be a continuous surjective 
	map between topological spaces. Then:
	\begin{enumerate}[leftmargin=*]
		\item \label{Lem.top1} If $X_0$ is dense in $X$, then $f(X_0)$ is dense in $Y$
		\item \label{Lem.top2} If $Y_0$ is dense in $Y$, $f$ is closed and 
							   $\overline{f^{-1}(Y_0)} =  f^{-1}(Y_1)$ for some $Y_1 \subseteq Y$, 
							   then $f^{-1}(Y_0)$ is dense in $X$.
	\end{enumerate}
\end{lemma}

\begin{proof}[Proof of Lemma~\ref{Lem.top}]
	\eqref{Lem.top1}: Note that
						$Y = f(X) = f(\overline{X_0}) \subseteq \overline{f(X_0)} \subseteq Y$.
						
	\eqref{Lem.top2}: Note that $Y_0 \subseteq f(\overline{f^{-1}(Y_0)}) = Y_1$. Since $f$ is closed and $Y_0$
	is dense in $Y$, we get $Y_1 = Y$, i.e.~$X = f^{-1}(Y) = \overline{f^{-1}(Y_0)}$.
\end{proof}

%

\begin{corollary}
	\label{Cor.Geometric_quotient_density}
	Let $H$ be a finite group that acts on a variety $X$ such that the geometric
	quotient  $\eta \colon X \to X/H \eqqcolon Y$  exists (e.g.~$X$ is quasi-projective).
	If $Y_0$ is dense in $Y$, then $\eta^{-1}(Y_0)$ is dense in $X$.
\end{corollary}

\begin{proof}
	Since $\eta^{-1}(Y_0)$ is $H$-stable, its closure $\overline{\eta^{-1}(Y_0)}$ in $X$ is $H$-stable as well, i.e.~there
	exists a subset $Y_1 \subseteq Y$ such that $\overline{\eta^{-1}(Y_0)} = \eta^{-1}(Y_1)$. The claim follows now from
	Lemma~\ref{Lem.top}\ref{Lem.top2}.
\end{proof}

\begin{lemma}
	\label{Lem.Density_and_conjugation_classes}
	Let $X' \subseteq X \subseteq G$, where $X'$ is dense in $X$.
	Then for any closed subgroup $H \subseteq G$, $C_G^H(X')$ is dense in $C_G^H(X)$.
\end{lemma}

\begin{proof}
	Consider the morphism $f \colon H \times G \to G$ given by $(h, g) \mapsto hgh^{-1}$.
	Then we get $C_G^H(X) = f(H \times X) = f(\overline{H \times X'}) \subseteq \overline{f(H \times X')} = \overline{C_G^H(X')}$.
\end{proof}



\begin{corollary}
	\label{Cor.reduction_to_conjugacy_invariant_subset}
	If $X \subseteq G$ is a subset and $X_{\tor}$ is dense in $X$,
	then $(\overline{C_G(X)})_{\tor}$ is dense in $\overline{C_G(X)}$.
\end{corollary}

\begin{proof}[Proof of Corollary~\ref{Cor.reduction_to_conjugacy_invariant_subset}]
	We apply Lemma~\ref{Lem.Density_and_conjugation_classes} to $X' \coloneqq X_{\tor} \subseteq X \subseteq G = H$
	in order to get that $C_G(X_{\tor}) = C_G(X)_{\tor}$ is dense in $C_G(X)$. 
	Hence the statement follows.
\end{proof}

\begin{corollary}
	\label{Cor.Properties_of_SVV}
	Let $X \subseteq G$ be a special subvariety. Then:
	
	\begin{enumerate}[leftmargin=*]
	\item \label{Cor.Properties_of_SVV_1} $X_{\tor}$ is dense in $X$;
	in particular, special subvarieties of $G$ satisfy the assumptions of Theorem~\ref{Thm.MMalg}.
	
	\item  \label{Cor.Properties_of_SVV_2} If $p \colon G \to G'$ is an isogeny of algebraic groups,
	 then the image $p(X)$ is a special subvariety of $G'$.
	\end{enumerate}
\end{corollary}

\begin{proof}
	Let $S \subseteq G$ be a commutative connected subgroup and let $t \in Z_G(S)_{\tor}$
	such that $X = \overline{C_G(tS)}$ and  $S_{\tor}$ is dense in $S$.
	
	\eqref{Cor.Properties_of_SVV_1}: Corollary~\ref{Cor.reduction_to_conjugacy_invariant_subset}
		implies that $(\overline{C_G(tS)})_{\tor}$ is dense in $\overline{C_G(tS)}$,
		since $t S_{\tor}$ is dense in $tS$.
		
	\eqref{Cor.Properties_of_SVV_2}: $p(S)$ is a closed commutative connected subgroup of $G'$.
	Moreover, since $p$ is an isogeny, we get $p(S)_{\tor} = p(S_{\tor})$ and this set
	is dense in $p(S)$ by Lemma~\ref{Lem.top}\eqref{Lem.top1}. 
	Using that $t \in Z_G(S)_{\tor}$, we get $p(t) \in Z_{G'}(p(S))_{\tor}$. Note that $p(C_G(tS)) = C_{G'}(p(t)p(S))$.
	Using that $p(\overline{Y}) = \overline{p(Y)}$ for any subset $Y \subseteq G$ (here we use that $p$ is closed),
	the claim follows.
\end{proof}

\subsection{Intersection of special subvarieties}
\label{Subsect.Intersection_SUV}

The following example shows that the intersection of two special subvarieties of a connected algebraic group
need not be the finite union of special subvarieties:

\begin{example}
	\label{Exa.Intersection_of_SUV is_not_SUV}
	Let $G = \GL_2(\kk)$ and for $k \geq 2$ consider
	\[
		S_k \coloneqq \Bigset{
		A_{k, \lambda}}{\lambda \in \kk^\ast} \subseteq 
		T \coloneqq 
		\Bigset{
		\begin{pmatrix}
			\lambda_1 & 0 \\
			0 & \lambda_2
		\end{pmatrix}
		}{\lambda_1, \lambda_2 \in \kk^\ast}
		\subseteq \GL_2(\kk) \, ,
	\]
	where
	\[
		A_{k, \lambda} = \begin{pmatrix}
					\lambda^k & 0 \\
					 0 & \lambda
				\end{pmatrix} \in \GL_2(\kk) \, .
	\]
	The algebraic quotient by conjugation on $\GL_2(\kk)$ is then given by 
	\[
		\pi \colon \GL_2(\kk) \to Z \coloneqq \AA^1 \times (\AA^1 \setminus \{0\}) \, , \quad
		A \mapsto (\tr(A), \det(A))
		\, , 
	\]
	where $\tr(A)$ denotes the trace of $A$, see 
	\cite[II.3.3, Beispiel~2]{Kr1984Geometrische-Metho}. Moreover, denote by
	\[
		\rho \coloneqq \pi |_T \colon T \to Z \, , \quad
		\begin{pmatrix}
			\lambda_1 & 0 \\
			0 & \lambda_2
		\end{pmatrix} \mapsto (\lambda_1 + \lambda_2, \lambda_1 \lambda_2)
	\]
	the restriction of $\pi$ to $T$.
	Since $\rho$ is finite (it is the geometric quotient by the
	$\ZZ/2\ZZ$-action that exchanges the two diagonal elements,
	see e.g.~Corollary~\ref{Cor.Restriction_of_alg_quotient_general}), it follows that
	$\rho(S_k)$ is an irreducible closed curve in $Z$.
	Hence,
	\begin{equation}
		\label{Eq.Example_Inclusion}
		\overline{C_{\GL_2(\kk)}(S_k)} \subseteq \pi^{-1}(\rho(S_k)) \, .
	\end{equation}
	On the other hand, we have $\lambda^k = \lambda$ if and only if $\lambda$ is a $(k-1)$-th root of unity.
	Thus if $\lambda^{k-1} \neq 1$, then the fibre $\pi^{-1}(\rho(A_{k, \lambda}))$
	consists of a single conjugacy class.
	Hence, the non-empty open subvariety
	\[
		U_k \coloneqq \pi^{-1}(\rho(\set{A_{k, \lambda}}{\lambda^{k-1} \neq 1}))
	\]
	of $\pi^{-1}(\rho(S_k))$ is contained in $\overline{C_{\GL_2(\kk)}(S_k)}$. Since the fibres of $\pi$ are all
	two-dimensional (Corollary~\ref{Cor.Conjugacy_classes_general}\eqref{Cor.Conjugacy_classes_general5}),
	and since $\pi^{-1}(\rho(S_k))$ is irreducible (Corollary~\ref{Cor.Preimage_irred_under_alg_quotient_general}), 
	it follows that the inclusion in~\eqref{Eq.Example_Inclusion} is in fact an equality.
	Moreover $C_{\GL_2(\kk)}(S_k) = U_k \cup S_k$ is a constructible, but non-open subset in its closure.
	Hence, $C_{\GL_2(\kk)}(S_k)$ is not locally closed in $\GL_2(\kk)$.
	
	Now, let $k_1, k_2 \geq 2$ be distinct. If $\lambda, \mu \in \kk^\ast$, then we have
	\[
		\rho(A_{k_1, \lambda}) = \rho \begin{pmatrix}
			\lambda^{k_1} & 0 \\
			0 & \lambda
		\end{pmatrix}
		= 
		\rho \begin{pmatrix}
					\mu^{k_2} & 0 \\
					0 & \mu
				\end{pmatrix}
		= \rho(A_{k_2, \mu})
	\]
	exactly when $\{ \lambda^{k_1}, \lambda \} = \{ \mu^{k_2}, \mu \}$. This last condition is equivalent to the fact that
	either $\mu = \lambda$, $\lambda^{k_1-k_2} = 1$ or $\mu = \lambda^{k_1}$, $\lambda^{k_1 k_2 -1} = 1$.
	This implies
	\begin{eqnarray*}
		\overline{C_{\GL_2(\kk)}(S_{k_1})} \cap \overline{C_{\GL_2(\kk)}(S_{k_2})}
		&=& \pi^{-1}(\rho(S_{k_1}) \cap \rho(S_{k_2})) \\
		&=& \coprod_{\substack{\lambda^{k_1 k_2-1} = 1 \ \textrm{or} \\ \lambda^{k_1-k_2} =1}} \pi^{-1}(\rho(A_{k_1, \lambda})) \, .
	\end{eqnarray*}
	Since $N \coloneqq \pi^{-1}(\rho(A_{k_1,1}))$ is two dimensional, but contains only the identity matrix
	$A_{k_1, 1}$ as a semi-simple $\GL_2(\kk)$-conjugacy class, the torsion points $N_{\tor}$ are not dense in $N$
	(see Corollary~
	\ref{Cor.Conjugacy_classes_general}\eqref{Cor.Conjugacy_classes_general2},\eqref{Cor.Conjugacy_classes_general3}).
	Thus the intersection $\overline{C_{\GL_2(\kk)}(S_{k_1})} \cap \overline{C_{\GL_2(\kk)}(S_{k_2})}$ 
	cannot be the finite union of special subvarieties in $\GL_2(\kk)$
	according to Corollary~\ref{Cor.Properties_of_SVV}\eqref{Cor.Properties_of_SVV_1}.
\end{example}

\subsection{The proof of the Manin-Mumford theorem for algebraic groups}

Before we give a proof of Theorem~\ref{Thm.MMalg} we establish the case when the 
connected algebraic group is
a certain product and reduce then Theorem~\ref{Thm.MMalg} to this special case.

\begin{prop}
	\label{Prop.MMProduct}
	Let $G = F \times L$ where $F$ is a connected commutative group and $L$ is a connected 
	affine reductive group.
	If $X \subseteq G$ is closed, stable under $G$-conjugation and $X_{\tor}$ is dense in $X$,
	then $X$ is a finite union of special subvarieties in $G$.
\end{prop}

\begin{proof}	
	Let $\pi_L \colon L \to L \aquot L$ be the algebraic quotient of the $L$-action by conjugation on $L$
	and let $\rho_L \coloneqq \pi_L |_T \colon T \to L \aquot L$ be its restriction to a
	maximal torus $T$ in $L$. Then $\rho_L$ is the geometric quotient of $T$ by conjugation 
	with the Weyl group of $L$ with respect to $T$, see Corollary~\ref{Cor.Restriction_of_alg_quotient_general}.
	Moreover, consider the morphisms
	\[
	\pi \coloneqq \id_{F} \times \pi_L \colon F \times L \to F \times (L \aquot L)
	\] 
	and
	\[
	\rho \coloneqq \id_{F} \times \rho_L \colon F \times T \to F \times (L \aquot L) \, .
	\]
	Let $Y \coloneqq \pi(X)$. We have
	\begin{equation}
		\label{Eq.Y}
		\rho^{-1}(Y) = X \cap (F \times T) \, , 
	\end{equation}
	since we get for all $f \in F$ by Corollary~\ref{Cor.Conjugacy_classes_general}\eqref{Cor.Conjugacy_classes_general6} that
	\begin{eqnarray*}
		\rho^{-1}(Y) \cap (\{f\} \times T) &=& \pi^{-1}(\pi(X \cap (\{f\} \times L))) \cap (\{f\} \times T) \\
		&=&  \{f\} \times \left( \pi_L^{-1}(\pi_L(\set{l \in L}{(f, l) \in X})) \cap T \right) \\
		&=& \{f\} \times (\set{l \in L}{(f, l) \in X} \cap T) \\
		&=& X \cap (\{f\} \times T)
	\end{eqnarray*}
	(we used the fact that $\set{l \in L}{(f, l) \in X}$ is stable under conjugation by $L$ and it is closed in $L$).
	Since $\rho_L \colon T \to L \aquot L$ is finite, the same folds for $\rho$. In particular
	\[
	Y \stackrel{\eqref{Eq.Y}}{=} \rho(X \cap (F \times T))
	\]
	is closed in $F \times T$.
	
	By Lemma~\ref{Lem.top}\eqref{Lem.top1} $\pi(X_{\tor})$ is dense in $Y$. 
	In particular, $Y \cap \pi(G_{\tor})$ is dense in $Y$.
	By Corollary~\ref{Cor.Geometric_quotient_density} applied to the geometric quotient 
	$\rho |_{\rho^{-1}(Y)} \colon \rho^{-1}(Y) \to Y$ and the dense subset $Y \cap \pi(G_{\tor})$ of $Y$,
	we get that $\rho^{-1}(Y) \cap \rho^{-1}(\pi(G_{\tor}))$ is dense in $\rho^{-1}(Y)$. 
	
	Note that $G_{\tor} = F_{\tor} \times C_L(T_{\tor})$ and therefore 
	$\pi(G_{\tor}) = F_{\tor} \times \rho_L(T_{\tor})$. Hence we get
	\begin{equation}
		\label{Eq.tor}
		\rho^{-1}(\pi(G_{\tor})) = 
		\rho^{-1}(F_{\tor} \times \rho_L(T_{\tor})) = 
		F_{\tor} \times T_{\tor} \, ,
	\end{equation}
	since $T_{\tor}$ is invariant under the Weyl group action
	by conjugation on $T$.

	Using~\eqref{Eq.Y}, \eqref{Eq.tor}, it follows that
	$X \cap (F_{\tor} \times T_{\tor})$ is dense in $X \cap (F \times T)$.
	By applying the
	Manin-Mumford Theorem for commutative algebraic groups 
	(see Theorem~\ref{Thm.MM}) to $X \cap (F \times T)$ 
	inside $F \times T$, 
	we get connected closed subgroups $S_1, \ldots, S_m$ of $F \times T$ and
	torsion points $h_1, \ldots, h_m \in F \times T$ with 
	\begin{equation}
		\label{Eq.MM_comm}
		X \cap (F \times T) = \bigcup_{i=1}^m h_i S_i
	\end{equation}
	such that $(S_i)_{\tor}$ is dense in $S_i$ for all $i$.
	Hence, $\overline{C_G(h_iS_i)}$ is a special subvariety of $G$ for all $i$. 
	
	Since $X$ is stable under $G$-conjugation and since $X$ is closed in $G$, we get
	\[
	X_{\tor} \subseteq C_G(X \cap (F \times T)) 
	\stackrel{\eqref{Eq.MM_comm}}{=} \bigcup_{i=1}^m C_G(h_i S_i)
	\subseteq \bigcup_{i=1}^m \overline{C_G(h_i S_i)} 
	\stackrel{\eqref{Eq.MM_comm}}{\subseteq} X \, .
	\]
	Since $X_{\tor}$ is dense in $X$, the result follows.
\end{proof}

Using Proposition~\ref{Prop.MMProduct} we can generalize to the case where $H$ is an arbitrary  connected affine algebraic group.

\begin{prop}
	\label{Prop.MMProduct_gen}
	Let $G = F \times H$ where $F$ is connected, commutative and $H$ is connected, affine.
	If $X \subseteq G$ is a closed irreducible subvariety that is stable under $G$-conjugation and 
	$X_{\tor}$ is dense in $X$,
	then $X$ is a special subvariety in $G$.
\end{prop}

For the proof we need the following lemma:

\begin{lemma}
	\label{Lem.Stratification}
	Let $H$ be a connected affine algebraic group with unipotent radical $V \coloneqq R_u(H)$.
	Write $H  = V L$, where $L$ is a reductive subgroup of $H$ (Levi-decomposition) and denote by
	$\pi \colon H \to L$ the surjection with kernel $V$ that restricts to the identity on $L$. 
	Then $H$ is the disjoint union of
	subsets $H_1, \ldots, H_m \subseteq H$ such that for all $1 \leq i \leq m$ we have: 
	\begin{enumerate}[leftmargin=*, label=\alph*)]
		\item \label{Lem.Stratification1} $H_i$ is locally closed in $H$.
		\item \label{Lem.Stratification2} $H_i$ is stable under $V$-conjugation.
		\item \label{Lem.Stratification3} The geometric quotient by $V$-conjugation on $H_i$ exists.
		\item \label{Lem.Stratification4} $(H_i)_{\tor}$ is contained in $C^V_{H}(H_i \cap L)$.
		\item \label{Lem.Stratification5} $C^V_{H}(H_i \cap L)$ is locally closed in $H$. 
		\item \label{Lem.Stratification6} $\pi|_{C_H^V(H_i \cap L)} \colon C_H^V(H_i \cap L) \to H_i \cap L$ is a geometric quotient
		for the $V$-conjugation.
	\end{enumerate}
\end{lemma}

\begin{proof}
	By a repeated application of a certain version  of Rosenlicht's theorem (see 
	\cite[Ch.~14, Theorem~5.3]{SaRi2017Actions-and-invari}
	and \cite[Proposition~1.6 and Remark~1.7(1)]{GrPf1993Geometric-Quotient})
	to $V$-conjugation on $H$, we may decompose $H$ into disjoint locally closed
	subsets
	\[
	H = \bigcup_{i=1}^m H_i
	\]	
	such that each $H_i$ is stable under $V$-conjugation, a geometric quotient
	$\pi_i \colon H_i \to Q_i$ for the $V$-conjugation on $H_i$ exists and
	$H_i$ is isomorphic to $\AA^{n_i} \times Q_i$ over $Q_i$ for some $n_i \geq 0$, 
	i.e.~there exists an isomorphism $\eta_i \colon \AA^{n_i} \times Q_i \to H_i$ 
	 such that $\pi_i \circ \eta_i \colon \AA^{n_i} \times Q_i \to Q_i$
	is the canonical projection. 
	Hence,~\ref{Lem.Stratification1}, ~\ref{Lem.Stratification2} and~\ref{Lem.Stratification3} hold.
	
	\ref{Lem.Stratification4}: If $t \in (H_i)_{\tor}$, then $t$ is contained in a maximal torus of $H$
	by \cite[Theorem~22.2]{Hu1975Linear-algebraic-g}. Since every maximal torus of $L$
	is a maximal torus in $H$ and since all maximal tori in $H$ are conjugate,
	there exists $h \in H$ such that $hth^{-1} \in L$. Writing $h = lv$ for $l \in L$ and $v \in V$
	gives $vtv^{-1} \in L$. Since $H_i$ is stable under $V$-conjugation, we get $vtv^{-1} \in L \cap H_i$
	
	\ref{Lem.Stratification5} and~\ref{Lem.Stratification6}:  
	Let $Z_i \coloneqq H_i \cap \pi^{-1}(H_i \cap L) \subseteq H$. Since $H_i$, $\pi^{-1}(H_i \cap L)$
	are both locally closed in $H$ and stable by $V$-conjugation, the same holds for $Z_i$.
	Since $H_i$ is isomorphic to $\AA^{n_i} \times Q_i$ over $Q_i$, it follows that
	$\pi_i(Z_i)$ is locally closed in $Q_i$ and the restriction
	$\pi_i |_{Z_i} \colon Z_i \to \pi_i(Z_i)$ is a geometric quotient by $V$-conjugation
	(we use that $\pi_i$ is a geometric quotient).
	Since the morphism $\pi |_{Z_i} \colon Z_i \to H_i \cap L$ is
	invariant under $V$-conjugation, there exists a morphism $\rho_i \colon \pi_i(Z_i) \to H_i \cap L$
	such that $\pi |_{Z_i} = \rho_i \circ \pi_i |_{Z_i}$. Hence,
	\[
	H_i \cap L \subseteq Z_i \xrightarrow{\pi_i |_{Z_i}} \pi_i(Z_i) \xrightarrow{\rho_i}
	H_i \cap L
	\]
	is the identity. This implies that $\pi_i(H_i \cap L)$ is closed in $\pi_i(Z_i)$
	and 
	\[
	\rho_i |_{\pi_i(H_i \cap L)} \colon \pi_i(H_i \cap L) \to H_i \cap L
	\]
	is an isomorphism. Thus $C^V_{H}(H_i \cap L) = \pi_i^{-1}(\pi_i(H_i \cap L))$ is closed in $\pi_i^{-1}(\pi_i(Z_i)) = Z_i$
	and the restriction $\pi |_{C^V_{H}(H_i \cap L)} \colon C^V_{H}(H_i \cap L) \to H_i \cap L$
	is a geometric quotient for the $V$-conjugation
	 (note that
	 $(\pi_i |_{Z_i}) \circ \eta_i |_{\AA^{n_i} \times \pi_i(Z_i)} \colon \AA^{n_i} \times \pi_i(Z_i) \to \pi_i(Z_i)$ is the canonical projection).
\end{proof}

\begin{proof}[Proof of Proposition~\ref{Prop.MMProduct_gen}]
	We use the notation of Lemma~\ref{Lem.Stratification}.
	Moreover, we denote $q \coloneqq \id_F \times \pi \colon F \times H \to F \times L$.

	Since $X$ is irreducible
	there exists a unique $j$ such that $X_1 \coloneqq X \cap (F \times H_j)$ is dense in $X$.
	Since $X_1$ is locally closed in $X$, it follows that $X_1$ is open in $X$. 
	Using that $X_{\tor}$ is dense in $X$, we get that
	$(X_1)_{\tor}$ is dense in $X_1$.
	
	Let
	\[
		X_0 \coloneqq X \cap (F \times C_H^V(H_j \cap L)) \subseteq F \times H \, .
	\]
	Since
	\[
		(X_1)_{\tor} \subseteq 
		X \cap (F \times (H_j)_{\tor}) 
		\stackrel{\textrm{Lem.~\ref{Lem.Stratification}\ref{Lem.Stratification4}}}{\subseteq}  X_0 \, ,
	\]
	it follows that $X_0$ is dense in $X$. By Lemma~\ref{Lem.Stratification}\ref{Lem.Stratification5}, $X_0$ is locally closed in $X$
	and hence $X_0$ is open and dense in $X$.
	
	Note that $(\overline{q(X)})_{\tor}$ is dense in $\overline{q(X)}$ by 
	Lemma~\ref{Lem.top}\eqref{Lem.top1} and that $\overline{q(X)}$ is stable under $F \times L$-conjugation
	(since $X$ is stable under $G$-conjugation). 
	An application of Proposition~\ref{Prop.MMProduct} gives now a connected
	commutative closed subgroup $S \subseteq F \times L$ and 
	$t \in Z_{F \times L}(S)_{\tor}$ such that $S_{\tor}$ is dense in $S$ and
	\begin{equation}
		\label{Eq.MMProduct}
		\overline{q(X)} = \overline{C_{F \times L}(tS)} \subseteq F \times L \, .
	\end{equation}
	
	Using that $q(X_0)$ is constructible and dense in $\overline{q(X)}$
	and that $C_{F \times L}(tS)$ is constructible and dense in $\overline{q(X)}$ (by~\eqref{Eq.MMProduct}),
	it follows that $q(X_0) \cap C_{F \times L}(tS)$ is dense in $C_{F \times L}(tS)$
	and also dense in $q(X_0)$. We get thus 
	\begin{eqnarray}
		\label{Eq.density_SUV}
		\overline{C_G^V\left(q(X_0)\right)} &\xlongequal{\textrm{Lem.~\ref{Lem.Density_and_conjugation_classes}}}&
		\overline{C_G^V\left(q(X_0) \cap C_{F \times L}(tS)\right)} \nonumber \\
		&\xlongequal{\textrm{Lem.~\ref{Lem.Density_and_conjugation_classes}}}&  \overline{C_G^V\left(C_{F \times L}(tS)\right)} = \overline{C_G(tS)} \, .
	\end{eqnarray}
	
	By Lemma~\ref{Lem.Stratification}\ref{Lem.Stratification6}, 
	$\pi|_{C_H^V(H_j \cap L)} \colon C_H^V(H_j \cap L) \to H_j \cap L$ is a geometric quotient
	for the $V$-conjugation. Thus, 
	\[
		q |_{F \times C_H^V(H_j \cap L)} \colon F \times C_H^V(H_j \cap L) \to F \times (H_j \cap L)
	\]
	is a geometric quotient for the $V$-conjugation. 
	Since $X$ is stable under $V$-conjuga\-tion,
	the same holds for the set $X_0$ in $F \times C_H^V(H_j \cap L)$. Therefore $X_0 = C_G^V(q(X_0))$.
	Hence,  $X = \overline{X_0} = \overline{C_G^V(q(X_0))} = \overline{C_G(tS)}$, where the last equality follows from~\eqref{Eq.density_SUV}. This finishes the proof.
\end{proof}

If $G$ is a connected algebraic group that satisfies the condition:
	\begin{equation}
	\label{Eq.Assumption_MM}
	\tag{$\ast$}
	\parbox{\dimexpr\linewidth-4em}{
		\strut
		there exists an isogeny $p \colon F \times H \to G$, where $F$ is connected and commutative,
		and $H$ is connected and affine,
		\strut
	}
\end{equation}
then it follows from Proposition~\ref{Prop.MMProduct_gen} together with Corollary~\ref{Cor.Geometric_quotient_density} and Corollary~\ref{Cor.Properties_of_SVV}\eqref{Cor.Properties_of_SVV_2} that the conclusion of Theorem~\ref{Thm.MMalg} holds.
However, the following example shows that not every connected algebraic group satisfies~\eqref{Eq.Assumption_MM}:

\begin{example}
	The following example is essentially due to Brion \cite[after Remarks~3.8]{Br2009Anti-affine-algebr}:
	
	Let $E$ be
	an indecomposable non-trivial extension of an elliptic curve $C$ 
	by $\GG_a$, i.e.~$E \not\simeq C \times \GG_a$,  see \cite[Proposition~11 and the Remark]{Ro1958Extensions-of-vect}  for the existence.
	Let $U$ be the group of upper triangular 
	unipotent $3 \times 3$-matrices and let $\GG_a \simeq Z \subseteq U$ be the center of $U$.
	Consider $G = E \times^{\GG_a} U$, where 
	$G$ denotes the quotient by the diagonal embedding $\GG_a \hookrightarrow  \GG_a \times Z$ into $E \times U$.
	We get then canonical natural embeddings $E \subseteq G$ and $U \subseteq G$ and in fact
	$E = G_{\ant}$, since $E$ is indecomposable.
	
	\begin{claim}
		The assumption~\eqref{Eq.Assumption_MM} is not satisfied for $G = E \times^{\GG_a} U$.
	\end{claim}
	
	Indeed, assume~\eqref{Eq.Assumption_MM} holds. Then there
	exist a central connected subgroup $F \subseteq G$ and a connected affine group $H \subseteq G$
	such that $F H = G$ and $F \cap H$ is finite. As the quotient $G/F$ is isomorphic to a quotient of $H$,
	it is affine. Hence, by definition $F \supseteq G_{\ant} = E$. If $F \supsetneq G_{\ant}$,
	then there exists $e \in E$ and $u \in U \setminus Z$ such that $eu \in F$. Hence there exists $u' \in U$ with
	$uu' \neq u'u$ in $U \subseteq G$ and therefore $euu' \neq  eu'u = u'eu$ in $G$, where the last equation follows from the
	fact that $E$ is central in $G$. This contradicts the fact that $F$ is central in $G$ and hence we get
	\[
	F = E = G_{\ant} \, .
	\]
	Consider
	\begin{equation}
		\label{Eq.Example_Brion}
		U/(U \cap H) \subseteq G/H \xleftarrow{\simeq} E/(E \cap H) \, .
	\end{equation}
	If $H \supseteq U$, then $\dim G/H \leq 1$ and thus 
	$E \cap H$ would not be finite, contradiction. Hence, $\dim U/ (U \cap H) > 0$.
	Since $U/(U \cap H)$ is affine and $E / E \cap H$ is not affine (the latter follows from the fact that $E$ is not affine and $E \cap H$ is finite), 
	we get from~\eqref{Eq.Example_Brion} 
	that  $\dim U/(U \cap H) = 1$. Note that $Z$ is not contained in $U \cap H$, since otherwise
	$F \cap H$ would not be finite. In summary we get  $U = Z (U \cap H)$.
	Using the fact that
	\[
	Z = \Bigset{ \begin{psmallmatrix}
			1 & 0 & b \\
			0 &  1 & 0 \\
			0 & 0 & 1
	\end{psmallmatrix}}{b \in \kk}
	\]
	it follows that
	\[
	U \cap H = \Bigset{ 
		\begin{psmallmatrix}
			1 & a & s(a, c) \\
			0 &  1 & c \\
			0 & 0 & 1
	\end{psmallmatrix} }{a, c \in \kk}
	\]
	for some morphism $s \colon \GG_a^2 \to \AA^1$.
	Since $U \cap H$ is a commutative group (that is isomorphic to $\GG_a^2$), it follows that
	$s(a, c) + s(x, z) + az = s(x, z) + s(a, c) + xc$ for all $a, c, x, z \in \kk$. This is absurd and hence we get a contradiction.
	This finishes the proof of the claim.
\end{example}

\begin{proof}[Proof of Theorem~\ref{Thm.MMalg}]
	Using Corollary~\ref{Cor.reduction_to_conjugacy_invariant_subset}, 
	we may assume that $X$ is stable under $G$-conjugation and closed in $G$, i.e.~$X = \overline{C_G(X)}$.
	
	Let $S \coloneqq (G_{\ant} \cap G_{\aff})^\circ \subseteq G$ and let
	\[
		\pi \colon  G_{\ant} \times G_{\aff} \to G_{\ant} \times^S G_{\aff} \, ,
	\]
	be the quotient of $G_{\ant} \times G_{\aff}$ by the subgroup
	\[
		\tilde{S} \coloneqq \set{(s^{-1}, s) \in G_{\ant} \times G_{\aff}}{s \in S} \simeq S \, .
	\] 
	Then the multiplication homomorphism
	$G_{\ant} \times G_{\aff} \to G$ is the composition of $\pi$ followed by a an isogeny 
	$p \colon G_{\ant} \times^S G_{\aff} \to G$. Note that $p^{-1}(G_{\tor})$ is the set of all torsion points
	in $G_{\ant} \times^S G_{\aff}$, since
	$p$ is an isogeny. By applying Corollary~\ref{Cor.Geometric_quotient_density} to $p^{-1}(X) \to X$
	and the dense subset $X_{\tor}$ of $X$, 
	it follows that the torsion points $p^{-1}(X)_{\tor} = p^{-1}(X_{\tor})$ 
	are dense in $p^{-1}(X)$.
	Using Corollary~\ref{Cor.Properties_of_SVV}\eqref{Cor.Properties_of_SVV_2}, it is enough to show that
	$p^{-1}(X)$ is a finite union of special subvarieties of $G_{\ant} \times^S G_{\aff}$. Hence, we may replace $G_{\ant} \times^S G_{\aff}$
	by $G$, i.e.~$G_{\ant} \cap G_{\aff}$ is connected. Then
	\[
		\pi \colon G' \coloneqq G_{\ant} \times G_{\aff} \to G
	\]
	is  surjective and its kernel is the central connected subgroup $\tilde{S}$ of $G'$
	
	As $\pi^{-1}(X)$ is closed in $G'$ and stable under $G'$-conjugation,
	the same holds for 
	\[
		X' \coloneqq \overline{\pi^{-1}(X)_{\tor}} \, .
	\] 
	As $X'_{\tor}$ is dense in $X'$, by Proposition~\ref{Prop.MMProduct_gen}, $X'$
	is a finite union of special subvarieties of $G'$, 
	i.e.~there exist connected commutative subgroups $S_1', \ldots, S_m'$ of $G'$ such that for all 
	$1 \leq i \leq m$ we have that $(S_i')_{\tor}$ is dense in $S_i'$ and there exists $t_i' \in Z_{G'}(S_i')_{\tor}$
	with
	\begin{equation}
		\label{Eq.union_above}
		X' = \overline{C_{G'}(t'_1 S'_1)} \cup \ldots \cup \overline{C_{G'}(t'_m S'_m)} \, .
	\end{equation}
	
	Let 
	\[
		U' \coloneqq \bigcup_{i=1}^m C_{G'}(t_i' S_i') \subseteq X' \, .
	\]
	Since $\tilde{S}$ is a connected commutative algebraic group, it is divisible (see Lemma~\ref{Lem.Comm_conn_group_div}).
	Since $\tilde{S}$ is central in $G'$, we get thus $\pi(G'_{\tor}) = G_{\tor}$. For $1 \leq i \leq m$, let
	$S_i \coloneqq \pi(S'_i)$, $t_i \coloneqq \pi(t_i') \in Z_{G}(S_i)_{\tor}$. Then $(S_i)_{\tor}$ is dense in $S_i$ by
	Lemma~\ref{Lem.top}\eqref{Lem.top1}. Thus $\overline{C_G(t_i S_i)}$ is a special subvariety of $G$ for all $1 \leq i \leq m$.
	We get now
	\begin{equation}
		\label{Eq.pi(U')_in_X}
		\pi(U') = \bigcup_{i=1}^m C_G(t_i S_i) \subseteq  \pi(X') \subseteq X \, .
	\end{equation}
	Since $U'$ is dense in $X'$ (by~\eqref{Eq.union_above}), 
	it follows that $\pi(U')$ is dense in $\pi(X')$ (see Lemma~\ref{Lem.top}\eqref{Lem.top1}).
	Note that
	\[
		X_{\tor} = X \cap \pi(G'_{\tor}) = \pi(\pi^{-1}(X)_{\tor}) 
		\subseteq \pi(X') \subseteq X \, .
	\]
	As $X_{\tor}$ is dense in $X$, we get thus that
	$\pi(X')$ is dense in $X$. In summary, we get that $\pi(U')$ is dense in $X$
	and using~\eqref{Eq.pi(U')_in_X} it follows that $X$ is equal to the union 
	of the special subvarieties $\overline{C_G(t_1 S_1)}, \ldots, \overline{C_G(t_m S_m)}$ of $G$.
\end{proof}

\section{A variation of the Manin-Mumford theorem}

\begin{notation}
	\label{Notation.variation}
	Throughout this section $G$ denotes a connected algebraic group.
	Moreover, $R_u(G_{\aff}) \subseteq G$ denotes the unipotent radical of $G_{\aff}$,
	$\hat{G} \coloneqq  G/R_u(G_{\aff})$ denotes the quotient and $p \colon G \to \hat{G}$
	denotes the canonical projection. Further, $\pi_{\hat{G}} \colon \hat{G} \to \hat{G} \aquot \hat{G}$
	denotes the categorical quotient of $\hat{G}$ by conjugation (see Proposition~\ref{Prop.Alg_quotient_general})
	and we consider the composition: 
	\[
		\pi \colon G \xrightarrow{p} \hat{G} \xrightarrow{\pi_{\hat{G}}} \hat{G} \aquot \hat{G} \, .
	\]
	We fix a maximal torus $T \subseteq G_{\aff}$. Then 
	$\hat{T} \coloneqq p(T)$ is a maximal torus in $\hat{G}_{\aff}$ and
	\[
		\rho \coloneqq \pi_{\hat{G}} |_{\hat{G}_{\ant} \hat{T}} \colon 
		\hat{G}_{\ant} \hat{T}  \to \hat{G} \aquot \hat{G} \, .
	\]
	is the geometric quotient of $\hat{G}_{\ant} \hat{T}$ with regard to conjugation by the Weyl group 
	$W = N_{\hat{G}_{\aff}}(\hat{T})/\hat{T}$ of $\hat{G}_{\aff}$ with respect to $\hat{T}$,
	see Corollary~\ref{Cor.Restriction_of_alg_quotient_general}.
\end{notation}

Remember that in general the intersection of special subvarieties (cf. Definition~\ref{Def.Special_subvarieties})
is not the finite union of special subvarieties, see Example~\ref{Exa.Intersection_of_SUV is_not_SUV}.
In order to remedy this failure, we introduce the following variant of torsion (which is coarser than the classical notion of torsion):

\begin{definition}
	\label{Def.utorsion}
	An element $g \in G$ is called \emph{unipotent-torsion}
	if there exists $n > 0$ such that $g^n$ is a unipotent element in $G$. If 
	$X \subseteq G$ is any subset, we denote by $X_{\utor}$ the unipotent-torsion elements from
	$G$ that belong to $X$.
\end{definition} 

Thus an element $\hat{g} \in \hat{G}$ is unipotent-torsion if and only if its semisimple part $\hat{g}_s$ is torsion.
Clearly every torsion element is unipotent-torsion. 
Moreover, we have the following relation between torsion and unipotent-torsion:

\begin{lemma}
	\label{Lem.Properties.utorsion} $ $
	\begin{enumerate}[leftmargin=*]
		\item \label{Lem.Properties.utorsion0} We have $p^{-1}(\hat{G}_{\utor}) = G_{\utor}$.
		
		\item \label{Lem.Properties.utorsion1} We have $\pi^{-1}(\pi(G_{\utor})) = G_{\utor} $, 
		$\rho^{-1}(\rho((\hat{G}_{\ant}\hat{T})_{\tor})) = (\hat{G}_{\ant}\hat{T})_{\tor}$
		and $\pi(G_{\utor}) = \rho((\hat{G}_{\ant}\hat{T})_{\tor})$.
		\item \label{Lem.Properties.utorsion2} If $X \subseteq G$ is a subset, 
		then $\rho^{-1}(\pi(X_{\utor}))= \rho^{-1}(\pi(X))_{\tor}$.
		\item \label{Lem.Properties.utorsion3} If $Y \subseteq \hat{G}_{\ant}\hat{T}$ 
		is a subset, then  $\pi^{-1}(\rho(Y_{\tor})) = \pi^{-1}(\rho(Y))_{\utor}$.
		%
		%
	\end{enumerate}
\end{lemma}

\begin{proof}
	\eqref{Lem.Properties.utorsion0}: Clearly $p(G_{\utor}) \subseteq \hat{G}_{\utor}$.
	Now, let $g \in G$ with $p(g) \in \hat{G}_{\utor}$. Then
	there exists $n > 0$ such that $p(g)^n$ is unipotent. Let $U \subseteq \hat{G}$ be the closure of the subgroup
	generated by $p(g)^n$. Hence, $p^{-1}(U)$ is unipotent as it is an extension of 
	$R_u(G_{\aff})$ by the unipotent group $U$. 
	In particular $g^n \in p^{-1}(U)$ is unipotent, i.e.~$g \in G_{\utor}$
	
	\eqref{Lem.Properties.utorsion1}: 
	Let $g \in G$ such that $\pi(g) = \pi(h)$, where $h \in G_{\utor}$. 
	By Corollary~\ref{Cor.Conjugacy_classes_general}\eqref{Cor.Conjugacy_classes_general4},
	there exists a unipotent $u \in \hat{G}$ and $a \in \hat{G}$ such that $p(g) = a u p(h)_s a^{-1}$ 
	and $u, p(h)_s$ commute.
	Since $h$ is unipotent-torsion, it follows that $p(h)_s$ is torsion and hence $p(g) \in \hat{G}_{\utor}$.
	By~\eqref{Lem.Properties.utorsion0} it follows that $g \in G_{\utor}$. This shows that
	$\pi^{-1}(\pi(G_{\utor})) = G_{\utor}$.
	
	The equality $\rho^{-1}(\rho((\hat{G}_{\ant}\hat{T})_{\tor})) = (\hat{G}_{\ant}\hat{T})_{\tor}$ follows from the fact, that $\rho$ is the geometric quotient
	of $\hat{G}_{\ant}\hat{T}$ by conjugation with the Weyl group, see 
	Corollary~\ref{Cor.Restriction_of_alg_quotient_general}.
	
	By~\eqref{Lem.Properties.utorsion0}, $p(G_{\utor}) = \hat{G}_{\utor}$.
	We get thus the last assertion from
	\[
	\pi(G_{\utor}) = \pi_{\hat{G}}(p(G_{\utor})) = \pi_{\hat{G}}(\hat{G}_{\utor}) = \pi_{\hat{G}}(\hat{G}_{\tor}) = \rho((\hat{G}_{\ant}\hat{T})_{\tor})
	\, ,
	\]
	where the third equality follows from the fact that $\pi_{\hat{G}}(a) = \pi_{\hat{G}}(a_s)$
	for all $a \in \hat{G}$  (see Corollary~\ref{Cor.Conjugacy_classes_general}\eqref{Cor.Conjugacy_classes_general3})
	and the last equality follows from Lemma~\ref{Lem.torsion_points}.
	
	\eqref{Lem.Properties.utorsion2} and \eqref{Lem.Properties.utorsion3}: These are direct consequences 
	of~\eqref{Lem.Properties.utorsion1}.
\end{proof}

The following consequence of Lemma~\ref{Lem.Properties.utorsion} 
shows that there is no hope for a simple description of all 
closed (irreducible) subsets of $G$ with dense unipotent-torsion points.

\begin{corollary}
	\label{Cor.X_utor_dense}
	Let $X \subseteq G$ be a closed irreducible subset such that 
	\begin{enumerate}[leftmargin=*, label=\alph*)]
		\item \label{Cor.X_utor_dense1} $\pi(X_{\utor})$ is dense in $\pi(X)$ or
		\item \label{Cor.X_utor_dense2} $\pi(X)$ is dense in $\hat{G} \aquot \hat{G}$.
	\end{enumerate}
	Then $X_{\utor}$ is dense in $X$.  In particular, $G_{\utor}$ is dense in $G$.
\end{corollary}

\begin{proof}
	Since $G_{\utor} = \pi^{-1}(\pi(G_{\utor}))$ (see Lem\-ma~\ref{Lem.Properties.utorsion}\eqref{Lem.Properties.utorsion1}), 
	we get
	$\pi(X_{\utor}) = \pi(G_{\utor}) \cap \pi(X)$.
	Let $X' \coloneqq \overline{X_{\utor}} \subseteq X$.
	
	Assume first that~\ref{Cor.X_utor_dense2} holds.
	Since $\pi(X)$ is constructible and dense in $\hat{G} \aquot \hat{G}$ 
	and since $\pi(G_{\utor}) = \rho((\hat{G}_{\ant}\hat{T})_{\tor})$ is dense in $\hat{G} \aquot \hat{G}$
	(see Lemma~\ref{Lem.Properties.utorsion}\eqref{Lem.Properties.utorsion1}), it follows that $\pi(X_{\utor})$ 
	is dense in $\hat{G} \aquot \hat{G}$. 
	Hence $\pi(X_{\utor})$ is dense in $\pi(X)$, i.e.~we may and will assume that \ref{Cor.X_utor_dense1} holds.
	As $X$ is non-empty ($X$ is irreducible), it follows that $X_{\utor}$ is non-empty as well.
	
	There exist an open dense subset 
	$U \subseteq \overline{\pi(X)} = \overline{\pi(X')}$ and  $d, d' \geq 0$ such that 
	the fibre of $\pi |_X \colon X \to \overline{\pi(X)}$ and $\pi |_{X'} \colon X' \to \overline{\pi(X')}$ 
	over each point of $U$ has dimension $d$ and $d'$, respectively.
	However, for all $a \in \pi(X_{\utor})$ we have that 
	\[
	(\pi |_X)^{-1}(a) = \pi^{-1}(a) \cap X =  \pi^{-1}(a) \cap X_{\utor} = (\pi |_{X'})^{-1}(a)
	\]
	and thus we conclude $d = d'$ (here we use that $X_{\utor}$ is non-empty). 
	This implies that $\dim X' = \dim X$. Since $X$
	is irreducible, we get  that $X_{\utor}$ is dense in $X$.
\end{proof}

In order to adjust the notion of special subvarieties for unipotent-torsion, 
let us reconsider Example~\ref{Exa.Intersection_of_SUV is_not_SUV}

\begin{example}
	Let $G \coloneqq \GL_2(\kk)$ and as in Example~\ref{Exa.Intersection_of_SUV is_not_SUV}, consider
	\[
	S_k \coloneqq \Bigset{
		A_{k, \lambda}}{\lambda \in \kk^\ast} \subseteq 
	T \coloneqq 
	\Bigset{
		\begin{pmatrix}
			\lambda_1 & 0 \\
			0 & \lambda_2
		\end{pmatrix}
	}{\lambda_1, \lambda_2 \in \kk^\ast}
	\subseteq \GL_2(\kk) \, ,
	\]
	where
	\[
	A_{k, \lambda} = \begin{pmatrix}
		\lambda^k & 0 \\
		0 & \lambda
	\end{pmatrix} \in \GL_2(\kk) \, .
	\]
	Then we have (by setting $k_1 = 3$, $k_2 = 5$ in Example~\ref{Exa.Intersection_of_SUV is_not_SUV}):
	\[
		\overline{C_{\GL_2(\kk)}(S_{3})} \cap \overline{C_{\GL_2(\kk)}(S_{5})}
		= \coprod_{\lambda^{14} = 1} \pi^{-1}(\rho(A_{3, \lambda})) \, .
	\]
	Now, we make a case-by-case analysis (by using Corollary~\ref{Cor.Conjugacy_classes_general}\eqref{Cor.Conjugacy_classes_general4}):
	\begin{enumerate}[leftmargin=*, label=\alph*)]
		\item \label{Exa.Case_Id} If $\lambda = 1$, then $\pi^{-1}(\rho(A_{3, \lambda})) = 
		\{ \begin{psmallmatrix}
			1 & 0 \\ 0 & 1
		\end{psmallmatrix} \} \cup
		C_G(\begin{psmallmatrix}
			1 & 1 \\ 0 & 1
		\end{psmallmatrix})$.
		\item \label{Exa.Case_-Id} If $\lambda = -1$, then $\pi^{-1}(\rho(A_{3, \lambda})) = 
		\{ \begin{psmallmatrix}
			-1 & 0 \\ 0 & -1
		\end{psmallmatrix} \} \cup 
		C_G(\begin{psmallmatrix}
			-1 & 1 \\ 0 & -1
		\end{psmallmatrix})$.
		\item \label{Exa.Case_Gen} If $\lambda^2 \neq 1$, then $\pi^{-1}(\rho(A_{3, \lambda})) = 
		C_G(\begin{psmallmatrix}
			\lambda^3 & 0 \\ 0 & \lambda
		\end{psmallmatrix})$.
	\end{enumerate}
	If $\lambda$ is a torsion point , then in~\ref{Exa.Case_Gen}, we get only torsion points,
	whereas in~\ref{Exa.Case_Id},~\ref{Exa.Case_-Id} we get only one torsion point. 
	However, all elements in the cases~\ref{Exa.Case_Id}, \ref{Exa.Case_-Id}, \ref{Exa.Case_Gen} 
	are unipotent-torsion if $\lambda$ is a torsion point. Hence, the unipotent-torsion points ly
	dense in the intersection $\overline{C_{\GL_2(\kk)}(S_{3})} \cap \overline{C_{\GL_2(\kk)}(S_{5})}$.
\end{example}

\begin{definition}
	\label{Def.uSUV}
	For any commutative connected subgroup $S \subseteq G$ and $h \in G_{\utor}$ 
	that commutes with $S$ we call 
	\[
		\pi^{-1}(\pi(hS))
	\]
	a \emph{unipotent-special subvariety} of $G$.
\end{definition}

We give a characterization of unipotent-special subvarieties which turn out to be very useful 
in the application.

\begin{lemma}
	\label{Lem.uSUV}
	The unipotent-special subvarieties of $G$ are the subvarities
	\[
		\pi^{-1}(\rho(h' S')) \, ,
	\]
	where $h' \in (\hat{G}_{\ant}\hat{T})_{\tor}$ and $S' \subseteq \hat{G}_{\ant}\hat{T}$ 
	is a closed connected subgroup. In particular, $G$ is a unipotent-special subvariety.
\end{lemma}

\begin{proof}
	Assume first that $X$ is a unipotent-special subvariety of of $G$. Then there exists a
	 commutative connected subgroup $S \subseteq G$ and $h \in G_{\utor}$ that commutes with $S$
	 such that $X = \pi^{-1}(\pi(hS))$.
	 Let $S_0 \coloneqq S_{\ant} (S_{\aff})_s$, where $(S_{\aff})_s$ denotes the closed subgroup
	 of semisimple elements in $S_{\aff}$ (note that $S_{\aff}$ is connected and commutative).
	 Then $S_0$ is a closed connected subgroup of $S$. By Lemma~\ref{Lem.Homomorphic_images} and
	 Remark~\ref{Rem.Homomorphic_images} we get  $p(S_{\ant}) \subseteq p(G_{\ant}) = \hat{G}_{\ant}$ and thus
	 \[
	 	S' \coloneqq p(S_0) = p(S_{\ant}) p((S_{\aff})_s) \subseteq \hat{G}_{\ant}p((S_{\aff})_s) \, . 
	 \]
	 As $h$ is unipotent-torsion, it follows that $h' \coloneqq p(h)_s$ is torsion. 
	 Since $h'$, $S'$ commute (by Lemma~\ref{Lem.Jordan}) and since 
	 $p((S_{\aff})_s)$ is a torus in $p(G_{\aff}) = \hat{G}_{\aff}$ we may assume after conjugation that
	 $h'$ and $S'$ both lie in  $\subseteq \hat{G}_{\ant} \hat{T}$. 
	 
	 Let $g \in S$. Then the elements $p(g)_u$, $p(g)_s$, $p(h)_u$, $h'$ commute pairwise by Lemma~\ref{Lem.Jordan}.
	 Since $p(g)_u$, $p(h)_u$ are unipotent and
	 since $p(g)_s$, $h'$ 
	 are semi-simple, it follows that $p(h)_u p(g)_u$ is unipotent and  $h' p(g)_s$ is semisimple. Hence we get
	 \begin{equation}
	 	\label{Eq.pi(hg)_decomp}
	 	\pi(hg) = \pi_{\hat{G}}( p(h)_u p(g)_u h' p(g)_s ) = \pi_{\hat{G}}( h' p(g)_s)
	 \end{equation}
	 by Corollary~\ref{Cor.Conjugacy_classes_general}\eqref{Cor.Conjugacy_classes_general4}. Since $S$
	 is commutative, $R_u(S_{\aff})$ consists of all unipotent elements of $S$
	 (see Lemma~\ref{Lem.Unipot_Radical}).
	 As $S_{\aff} = R_u(S_{\aff}) (S_{\aff})_s$, we get thus
	 \[
	 	p(S) = p(R_u(S_{\aff}))S'\quad \textrm{and} \quad p(R_u(S_{\aff})) \cap S' = 1   \, .
	 \]
	 Hence, there exist $a \in  p(R_u(S_{\aff}))$, $b \in S'$ such that $p(g) = ab$,
	 $a$ is unipotent, $b$ is semisimple and $a, b$ commute (since $S$ is commutative). By the uniqueness
	 of the Jordan decomposition (Lemma~\ref{Lem.Jordan}), it follows that $a = p(g)_u$ and $b = p(g)_s$.
	 In particular, $p(g)_s \in S'$. 
	 Now,~\eqref{Eq.pi(hg)_decomp} shows that $\pi(hS) = \rho(h' S')$ and hence, $X = \pi^{-1}(\rho(h' S'))$.
	 
	 \medskip
	 
	 Now, let $h' \in (\hat{G}_{\ant}\hat{T})_{\tor}$ and let $S' \subseteq \hat{G}_{\ant}\hat{T}$ 
	 be a closed connected subgroup. 
	 Consider the surjection 
	 \[
	 	q \coloneqq p |_{G_{\ant}T} \colon G_{\ant}T \to \hat{G}_{\ant} \hat{T} \, .
	 \]
	 Let $h \in G_{\ant} T$ with $q(h) = h'$ and let $S \coloneqq q^{-1}(S')$.
	 Then $h$ is unipotent-torsion (by Lemma~\ref{Lem.Properties.utorsion}\eqref{Lem.Properties.utorsion0}) 
	 and $S$ is a closed connected commutative subgroup of $G$
	 (since the kernel of $q$ is the unipotent subgroup $G_{\ant}T \cap R_u(G_{\aff})$)
	  such that $h$ and $S$ commute. 
	 Then we have $p(hS) = h' S'$ and thus
	 $\pi^{-1}(\pi(hS)) = \pi^{-1}(\rho(h' S'))$ is a unipotent-special subvariety of $G$.
\end{proof}

\begin{remark}
	\label{Rem.uSUV}
	Lemma~\ref{Lem.uSUV} implies that $\pi(X)$ 
	is closed in $\hat{G} \aquot \hat{G}$ and $X$ is closed in $G$ for every unipotent-special subvariety $X$ of $G$.
\end{remark}

%

In the following proposition we list several properties of unipotent-special subvarieties. In particular,
it follows from~\eqref{Prop.uSUV3} below that the set of finite unions of unipotent-special subvarieties form 
a bounded distributive lattice under union and intersection.

\begin{prop} $ $ 
	\label{Prop.uSUV}
	\begin{enumerate}[leftmargin=*]
		\item \label{Prop.uSUV1} Every unipotent-special subvariety of $G$ is irreducible and closed in $G$.
		\item \label{Prop.uSUV2} If $X$ is a unipotent-special subvariety of $G$, then $X_{\utor}$ is dense in $X$.
		\item \label{Prop.uSUV3} If $X_1$, $X_2$ are unipotent-special subvarieties of $G$, then
		$X_1 \cap X_2$ is a finite union of unipotent-special subvarieties.
		\item \label{Prop.uSUV4} 
		If $X$ is a special subvariety of $G$, then $\pi^{-1}(\pi(X))$ is a unipotent-special subvariety of $G$.
	\end{enumerate}
\end{prop}

For the proof of~\eqref{Prop.uSUV3}, we need the following lemma:

\begin{lemma}
	\label{Lem.Subtori_I}
	Let $H$ be a connected commutative  algebraic group, let $S_1, S_2$ be connected closed subgroups of $H$
	and let $h_1, h_2 \in H$ be torsion points. 
	If the intersection $h_1 S_1 \cap h_2 S_2$ is non-empty, then $h_1 S_1 \cap h_2 S_2 = h_3 (S_1 \cap S_2)$ for some
	torsion point $h_3 \in H$. 
\end{lemma}

\begin{proof}
	Without loss of generality we may assume that $h_1 = 1$. Let $x \in S_1 \cap h_2 S_2$.
	Then $x^m \in S_1 \cap S_2 \eqqcolon S$ for some $m > 0$ (that satisfies $h_2^m = 1$). By multiplying $m$ with 
	the index of the connected component $S^\circ$ in $S$ we may assume that $x^m \in S^\circ$. 
	Since $S^\circ$ is divisible (Lemma~\ref{Lem.Comm_conn_group_div}) there exists 
	$s \in S^\circ$ with $s^m = x^m$. Then $h_3 \coloneqq xs^{-1}  \in H$ is a torsion point that belongs
	to $S_1 \cap h_2 S_2$. We have now $h_3 S = S_1 \cap h_2 S_2$.
\end{proof}

\begin{proof}[Proof of Proposition~\ref{Prop.uSUV}]
	\eqref{Prop.uSUV1}: Let $X \subseteq G$ be a unipotent-special subvariety. 
	By Remark~\ref{Rem.uSUV} $\pi(X)$ is closed in $\hat{G} \aquot \hat{G}$
	and it is irreducible. The statement follows thus from Corollary~\ref{Cor.Preimage_irred_under_alg_quotient_general}.
	
	\eqref{Prop.uSUV2}: Write $X = \pi^{-1}(\rho(hS))$ where $h \in (\hat{G}_{\ant}\hat{T})_{\tor}$
	and $S \subseteq \hat{G}_{\ant}\hat{T}$ is a closed connected subgroup, 
	see Lemma~\ref{Lem.uSUV}.
	Then $hS_{\tor}$ is dense in $hS$ (see Proposition~\ref{Prop.dense}) 
	and hence $\rho(hS_{\tor})$ is dense in $\rho(hS) = \pi(X)$ (see Lemma~\ref{Lem.top}\eqref{Lem.top1}).
	As $p^{-1}(hS_{\tor})$ is contained in $X_{\utor}$ 
	(see Lemma~\ref{Lem.Properties.utorsion}\eqref{Lem.Properties.utorsion0}), 
	we conclude that $\pi(X_{\utor})$ is dense in $\pi(X)$. 
	Now, Corollary~\ref{Cor.X_utor_dense} implies~\eqref{Prop.uSUV2}.
	
	\eqref{Prop.uSUV3}: For $i=1, 2$, let $S_i \subseteq \hat{G}_{\ant}\hat{T}$ 
	be a closed connected subgroup and let
	$h_i \in (\hat{G}_{\ant}\hat{T})_{\tor}$ such that $X_i = \pi^{-1}(\rho(h_i S_i))$, see Lemma~\ref{Lem.uSUV}.  
	Then
	\[
	\rho^{-1}(\rho(h_1 S_1)) \cap \rho^{-1}(\rho(h_2 S_2)) = 
	\bigcup_{w_1, w_2 \in W} w_1 h_1 S_1 w_1^{-1} \cap w_2 h_2 S_2 w_2^{-1} \, ,
	\]
	By Lemmas~\ref{Lem.Subtori_I} and~\ref{Lem.Conn_Comp_Comm} there exist
	commutative connected closed subgroups $A_1, \ldots, A_n$ of $\hat{G}_{\ant}\hat{T}$  
	and $b_1, \ldots, b_n \in (\hat{G}_{\ant}\hat{T})_{\tor}$ such that  
	\[
	\rho^{-1}(\rho(h_1 S_1)) \cap \rho^{-1}(\rho(h_2 S_2)) = \bigcup_{i=1}^n b_i A_i \, .
	\]
	Hence $\rho(h_1 S_1) \cap \rho(h_2 S_2) = \bigcup_i \rho(b_i A_i)$ and thus 
	$X_1 \cap X_2 = \bigcup_i \pi^{-1}(\rho(b_i A_i))$. This implies~\eqref{Prop.uSUV3} by
	using Lemma~\ref{Lem.uSUV}.
	
	\eqref{Prop.uSUV4}: There exist a commutative connected closed subgroup $S \subseteq G$ such that $S_{\tor}$ 
	is dense in $S$ and a $t \in G_{\tor}$ that commutes with $S$ such that $X = \overline{C_G(tS)}$. 
	By definition, $\pi^{-1}(\pi(tS))$ is a unipotent-special subvariety.
	By Remark~\ref{Rem.uSUV}, it follows that $\pi(tS)$ is closed in  $\hat{G} \aquot \hat{G}$. 
	This implies that $\pi(tS) \subseteq \pi(X) \subseteq \overline{\pi(C_G(tS))}
	= \pi(tS)$. Hence $\pi(X) = \pi(tS)$ and thus $\pi^{-1}(\pi(X))$ is a unipotent-special subvariety of $G$.
\end{proof}

\begin{remark}
	If $R_u(G_{\aff})$ is trivial, i.e.~$G =\hat{G}$, then for every subset $X \subseteq G$ that is invariant under
	$G$-conjugation we have that
	\[
		\pi^{-1}(\pi(X)) = \set{ u x_s}{ \textrm{$x\in X$, $u \in G$ is unipotent and $u x_s = x_s u$}}
	\]
	by Corollary~\ref{Cor.Conjugacy_classes_general}\eqref{Cor.Conjugacy_classes_general4}. Thus
	$\pi^{-1}(\pi(X))$ arises from $X$ by multiplying the semisimple part of every element in $X$ with all 
	unipotent elements in $G$ that commute with the semi-simple part.
\end{remark}

\begin{theorem}
	\label{Thm.MMalg_unipot}
	If $X \subseteq G$ is a subset such that $X_{\utor}$ is dense in $X$,
	then $\pi^{-1}(\overline{\pi(X)})$ is a finite union of unipotent-special subvarieties of $G$.
\end{theorem}

\begin{proof}
	Lemma~\ref{Lem.Properties.utorsion}\eqref{Lem.Properties.utorsion2} gives
	$\rho^{-1}(\pi(X))_{\tor} = \rho^{-1}(\pi(X_{\utor}))$. Now, Lemma~\ref{Lem.top}\eqref{Lem.top1} and Corollary~\ref{Cor.Geometric_quotient_density} imply that
	$\rho^{-1}(\pi(X))_{\tor}$ is dense in $\rho^{-1}(\pi(X))$. Since $\pi(X)$ is dense in $\overline{\pi(X)}$,
	it follows that $\rho^{-1}(\pi(X))$ is dense in $\rho^{-1}(\overline{\pi(X)})$, cf.~Corollary~\ref{Cor.Geometric_quotient_density}. Hence $\rho^{-1}(\overline{\pi(X)})_{\tor}$
	is dense in $\rho^{-1}(\overline{\pi(X)})$. By Theorem~\ref{Thm.MM}, 
	there exist closed connected subgroups $S_1, \ldots, S_n \subseteq \hat{G}_{\ant} \hat{T}$
	and $h_1, \ldots, h_n \in (\hat{G}_{\ant} \hat{T})_{\tor}$ such that 
	$\rho^{-1}(\overline{\pi(X)}) = h_1 S_1 \cup \ldots \cup h_n S_n$. Hence,
	\[
		\pi^{-1}\left(\overline{\pi(X)}\right) = \pi^{-1}(\rho(h_1 S_1)) \cup \ldots \cup \pi^{-1}(\rho(h_n S_n)) \, .
	\]
	By Lemma~\ref{Lem.uSUV}, $\pi^{-1}(\rho(t_i S_i))$ is a unipotent-special subvariety for all $i$.
\end{proof}

\begin{remark}
	If $Y \subseteq G$ is a unipotent-special subvariety, then $\pi(Y)$ is closed in $\hat{G} \aquot \hat{G}$ by Remark~\ref{Rem.uSUV}.
	Hence, in order that the conclusion of Theorem~\ref{Thm.MMalg_unipot} holds, we are forced to take the closure $\overline{\pi(X)}$ in the expression~$\pi^{-1}(\overline{\pi(X)})$.
\end{remark}

\section{Canonical heights and Bogomolov's conjecture} 
In this section we assume that $\kk$ is the field of algebraic numbers $\overline{\Q}$.
Moreover, we use again Notation~\ref{Notation.variation} in this section, i.e.~$p \colon G \to \hat{G}$
denotes the quotient by the unipotent radical of $G$, $\pi_{\hat{G}} \colon \hat{G} \to \hat{G} \aquot \hat{G}$ denotes the categorical
quotient by conjugation, $\pi$ denotes the composition
\[
\pi \colon G \xrightarrow{p} \hat{G} \xrightarrow{\pi_{\hat{G}}} \hat{G} \aquot \hat{G}
\]
and $\hat{T}$ is a fixed maximal torus in $\hat{G}_{\aff}$.
In order to ease notation we set in this section 
\[ \hat{S} = \hat{G}_{\text{ant}}\hat{T}, 
\]
which is a semi-abelian variety defined over a number field. Moreover, let
$\rho \coloneqq \pi_{\hat{G}} |_{\hat{S}} \colon \hat{S} \to \hat{G} \aquot \hat{G}$. Recall that $\rho$ 
is the geometric quotient of
$\hat{S}$ by $W$-conjugation, where
$W$ denotes the Weyl group of $G_{\aff}$ associated to $\hat{T}$, 
see Corollary~\ref{Cor.Restriction_of_alg_quotient_general}. \\

We now introduce heights, which are functions on algebraic points that satisfy desirable arithmetic  properties \cite{hindrysilverman}. We say that a height satisfies the 
\emph{Northcott property} if the set of points of bounded degree (over $\Q$) and height is finite.

We recall that $\hat{S}$ is an extension of an abelian variety $A$ by  some algebraic torus 
$\GG_m^t, t \geq 0$. We denote by $\pi_A: \hat{S}\rightarrow A$ the canonical projection to $A$, 
which is a locally trivial principal
$\GG_m^t$-bundle (with respect to the Zariski topology), see~e.g.~\cite[Proposition~14]{Se1958Espaces-fibres-alg}. 
We compactify $\hat{S}$ smoothly
in a $\GG_m^t$-equivariant way  as in 
\cite[Sect.~1]{Hi1988Autour-dune-conjec} via  
\[
	\hat{S}_{\text{eq.}} = \hat{S}\times^{\mathbb{G}_m^t}(\mathbb{P}^1)^{t} \, .
\]
This turns $\hat{S}_{\text{eq.}}$ into a locally trivial $(\mathbb{P}^1)^{t}$-bundle over $A$ 
(with respect to the Zariski topology) and we denote this bundle again by $\pi_A \colon \hat{S}_{\text{eq.}} \to A$. 
We set $D_\infty$ to be the divisor at infinity $D_\infty = \hat{S}_{eq.}\setminus \hat{S}$ endowed
with the reduced scheme structure. For every integer $n \geq 2$, it holds that 
\[
	[n]^*D_\infty \ \sim \ nD_\infty \quad \textrm{are linearly equivalent} \, ,
\]
where $[n]$ is the multiplication by $n$ morphism on $\hat{S}$ \cite[Lemme~2(ii)]{Hi1988Autour-dune-conjec}. We set 
\[
	\hat{h}_\ell \colon \hat{S} \to \RR_{\geq 0} \, , \quad	 
	g \mapsto \lim_{k \rightarrow \infty}\frac{h_{D_\infty}([2^k]g)}{2^k} \, ,
\]
where $h_{D_\infty}$ is the Weil-height on $\hat{S}_{\text{eq.}}$ attached to $D_\infty$ (see \cite[Theorem~B.3.2]{hindrysilverman}). This limit exists 
and is non-negative, which can be shown using  the Néron-Tate limiting argument 
and the fact that $D_{\infty}$ is an effective divisor such that the base-locus of $|D|$ is outside of $\hat{S}$, 
\cite[Theorem~B.4.1, Theorem~B.3.2(e)]{hindrysilverman}. Note that if the forward orbit $\set{[2^k]g}{k \in \NN}$ under $[2]$ 
is finite for some $g \in \hat{S}$, then $\hat{h}_\ell(g) = 0$.

We also choose an ample symmetric divisor $L$ on $A$
(e.g. one can take any very ample divisor $H$ on $A$ and choose $L = H + [-1]^\ast H$). Then for any $n \geq 2$
\[
		[n]^*L \ \sim \ n^2 L \quad \textrm{are linearly equivalent} \, ,
\]
see e.g.~\cite[Corollary~A.7.2.5]{hindrysilverman}. We set
\[
 	\hat{h}_q \colon \hat{S} \to \RR_{\geq 0} \, , \quad 
 	g \mapsto \lim_{k \rightarrow \infty}\frac{h_{L}( [2^k] \pi_A(g) )}{ 4^{k}} \, ,
\]
where $h_{L}$ is the Weil-height on $A$ attached to $L$. This limit exists, is non-negative and $\hat{h}_q(g)$ vanishes 
if and only if $\pi_A(g)$ is a torsion point of $A$, \cite[Theorem~B.4.1, Proposition~B.5.3(a)]{hindrysilverman}. 

\begin{definition}\label{defheight1}
	We define the \emph{canonical height} $\hat{h}_{\hat{S}}$ on $\hat{S}$ to be 
	\[
		\hat{h}_{\hat{S}} = \hat{h}_q + \hat{h}_\ell \, .
	\]
\end{definition}

Using \cite[Theorem~B.4.1(i), Theorem~B.3.2(b),(c)]{hindrysilverman} it follows that
$\hat{h}_{\hat{S}}$ is the same as the Weil-height $h_{D_{\infty} + \pi_A^\ast(L)}$ up to a bounded function.
By \cite[Lemme~1]{Hi1988Autour-dune-conjec} the divisor $D_{\infty} + \pi_A^\ast(L)$ is ample. Thus
$h_{D_{\infty} + \pi_A^\ast(L)}$ (and hence $\hat{h}_{\hat{S}}$) 
satisfies the Northcott property \cite[Theorem~B.3.2(g)]{hindrysilverman}.

\begin{remark}	
	\label{Rem.Torsion_on_hatS}
	The height $\hat{h}_{\hat{S}}$ vanishes exactly on the torsion points of $\hat{S}$. Indeed, let $g \in \hat{S}$.
	If $g$ is a torsion point, then $\set{[2^k]g}{k \in \NN}$ is finite and thus $\hat{h}_q$, 
	$\hat{h}_\ell$ both vanish on $g$. On the other hand, if $\hat{h}_{\hat{S}}(g) = 0$, then $\pi_A(g)$ is a torsion point, i.e.~there exists $m \geq  2$ such that $[m]g \in \hat{T}$. By \cite[Theorem~B.3.2(b),(c),(d)]{hindrysilverman}
	there is a constant $C_m$ such that $|h_{D^\infty}([m]q) - m h_{D^\infty}(q) | < C_m$
	for all $q \in \hat{S}$. By replacing $q$ with 
	$[2^k]g$, dividing by $2^k$ and taking the limit we get thus
	\[
		\hat{h}_{\ell}([m]g) = \lim_{k \rightarrow \infty}\frac{h_{D_\infty}([m2^k]g)}{2^k}
		= m \hat{h}_{\ell}(g) = 0 \, .
	\] 
	Let $\iota \colon (\PP^1)^t \hookrightarrow \hat{S}_{\text{eq.}}$ be a closed $\GG_m^t$-equivariant
	immersion with image the fibre of $\pi_A$ over the neutral element of $A$. Note that
	\begin{equation}
		\label{Eq.NicePullback}
		[2]^\ast \iota^\ast D_{\infty} = \iota^\ast [2]^\ast D_{\infty} \ \sim \ 2 \iota^\ast D_{\infty} \, .
	\end{equation}
	Let $P \in (\PP^1)^t$ with $\iota(P) = [m]g$. There is a constant $C$ such that
	\[
		|h_{\iota^\ast D^\infty}([2^k]P) -  h_{D^\infty}([2^k] \iota(P)) | < C \quad \textrm{for all $k \geq 0$}
	\]
	by \cite[Theorem~B.3.2(b)]{hindrysilverman}. Dividing by $2^k$ and taking the limit gives:
	\[
		0 = \hat{h}_{\ell}([m]g) = \lim_{k \to \infty} \frac{h_{\iota^\ast D^\infty}([2^k]P)}{2^k} \, .
	\]
	Since $\iota^\ast D_{\infty}$ is (very) ample and using~\eqref{Eq.NicePullback}, it follows that
	$P$ is a torsion point of $\GG_m^t$ and hence, the same holds for 
	$g^m$, see~\cite[Propostion~B.4.2(a)]{hindrysilverman}.
\end{remark}


We now  note that it is easy to construct a  canonical height on $\hat{S}$ that is invariant under the action of the Weyl group of $\hat{G}_{\text{aff}}$.  
\begin{definition}\label{inv1}
	Let   $W$ be the Weyl-group of $\hat{G}_{\text{aff}}$ and let 
	$\hat{h}_{\hat{S}}$ be the canonical height constructed in Definition~\ref{defheight1}. We define 
	\[
		\hat{h}^W_{\hat{S}} \colon \hat{S} \to \RR_{\geq 0} \, ,
		\quad g \mapsto \sum_{w \in W}\hat{h}_{\hat{S}}(wgw^{-1}) \, .
	\]
\end{definition} 
Note that this is well-defined as $W$ is finite and that $\hat{h}_{\hat{S}}^W$ satisfies the Northcott property. 
We can now define a canonical height on $G$. 
\begin{definition}\label{defheight2} For $\hat{g} \in \hat{G}$, we define 
	$\hat{h}_{\hat{G}}(\hat{g}) = \hat{h}^W_{S}(\hat{s})$, where $\hat{s} \in \hat{S}$ is an element that is 
	conjugated to the semi-simple part $\hat{g}_s$ of $\hat{g}$, see Lemma~\ref{Lem.torsion_points} 
	for the existence of $\hat{s}$. 
	From Lemma~\ref{Lem.Jordan}, Corollary~\ref{Cor.Conjugacy_classes_general}\eqref{Cor.Conjugacy_classes_general1} 
	and Definition~\ref{inv1} it
	follows that this is well-defined. For $g \in G$ we then set 
	\[
		\hat{h}_G(g) = \hat{h}_{\hat{G}}(p(g)) \, .
	\] 
\end{definition}

\begin{remark}
	\label{Rem.Height}
	In other words, if $g \in G$ and $\hat{s} \in \rho^{-1}(\pi(g))$, then $\hat{h}_{G}(g) = \hat{h}_{\hat{G}}(\hat{s})$,
	since $\pi(g) = \pi_{\hat{G}}(\hat{g}) = \pi_{\hat{G}}(\hat{g}_s)$ for $\hat{g} = p(g)$, see Corollary~\ref{Cor.Conjugacy_classes_general}\eqref{Cor.Conjugacy_classes_general3}.
\end{remark}

The above constructed height function $\hat{h}_G \colon G \to \RR_{\geq 0}$ is invariant under $G$-conjugation and moreover:

\begin{lemma} 
	\label{Lem.Height_u_tor}
	The height function $\hat{h}_G \colon G \to \RR_{\geq 0}$ 
	vanishes exactly on the unipotent-torsion elements of $G$.  
\end{lemma}

\begin{proof} 
	If $g \in G$ is unipotent-torsion, then  $\hat{g} \coloneqq p(g)$ is unipotent-torsion in $\hat{G}$. 
	Lemma~\ref{Lem.Jordan} implies that $\hat{g} = \hat{g}_s\hat{g}_u$ with $\hat{g}_s$ being torsion. 
	Then any element $\hat{s} \in \hat{S}$ that is conjugated to $\hat{g}_s$ is torsion as well
	and in particular the forward-orbit $\set{\hat{s}^{2^k}}{k \in \NN}$ is finite. Thus for every $w \in W$, we have that
	$\hat{h}_q(w\hat{s}w^{-1})$ and $\hat{h}_{\ell}(w\hat{s}w^{-1})$ vanish and hence, $\hat{h}_{G}(g)$ vanishes as well.
	
	On the other hand, if $\hat{h}_G(g) = 0$ for some $g \in G$, then the semi-simple part $\hat{g}_s$ of 
	$\hat{g} \coloneqq p(g)$ is conjugated to a point $\hat{s} \in \hat{S}$ with $\hat{h}_{\hat{S}}(\hat{s}) = 0$. 
	As the points on $\hat{S}$, whose canonical height vanish are exactly the torsion points of $\hat{S}$ 
	(see Remark~\ref{Rem.Torsion_on_hatS}) we conclude the proof.  
\end{proof}

\begin{theorem}\label{bog} Let $X \subset G$ be an irreducible subset such that $\pi^{-1}(\overline{\pi(X)})$ is not a unipotent-special subvariety.  
	Then, there exists an $\epsilon  = \epsilon(X)> 0$, such that the set 
	\[
	\set{g \in \pi^{-1}(\overline{\pi(X)}) }{ \hat{h}_G(g) < \epsilon }
	\]
	is not dense in $\pi^{-1}(\overline{\pi(X)})$. 
\end{theorem}

Note that $\pi^{-1}(\overline{\pi(X)})$ is irreducible by Corollary~\ref{Cor.Preimage_irred_under_alg_quotient_general}.

\begin{proof}
	Since $\overline{\pi(X)}$
	is irreducible, the irreducible components of $\rho^{-1}(\overline{\pi(X)})$ have all the same dimension. Then none of these irreducible components is a torsion translate of a connected closed subgroup of $\hat{S}$, since otherwise $\overline{\pi(X)} = \rho(hS)$ for a torsion point $h \in \hat{S}$ and  a connected
	closed subgroup $S \subseteq \hat{S}$ and thus $\pi^{-1}(\overline{\pi(X)})$ would be a unipotent-special subvariety of $G$ by Lemma~\ref{Lem.uSUV}.
	
	Let $Y = \rho^{-1}(\overline{\pi(X)}))$.
 	From the Bogomolov conjecture for semi-abelian varieties proven in \cite{davidphilippon} (and by means of equidistribution in 
 	\cite[Proposition~6.1]{kuehnesmall} (see also \cite[\S3]{KuehneBHC})) follows that there exists $\varepsilon = \epsilon(X) > 0$ such that the set $Y_{\varepsilon} \coloneqq \set{y \in Y}{\hat{h}_{\hat{G}}(y) < \varepsilon}$ is nowhere dense in $Y$. 
 	Since $\hat{h}_{\hat{G}} |_{\hat{S}} = \hat{h}_{\hat{S}}^W$ is invariant under $W$-conjugation (see Definition~\ref{inv1}), 
 	we get $Y_{\varepsilon} = \rho^{-1}(\rho(Y_{\varepsilon}))$.
 	From Corollary~\ref{Cor.Geometric_quotient_density} 
 	it follows that $\rho(Y_{\varepsilon})$ is not dense in $\rho(Y)$ and with Lemma~\ref{Lem.top}\eqref{Lem.top1} that 
 	$\pi^{-1}(\rho(Y_\varepsilon))$ is not dense in $\pi^{-1}(\rho(Y)) = \pi^{-1}(\overline{\pi(X)})$. 
 	It remains to show that 
	\[
		\pi^{-1}(\rho(Y_\epsilon)) = \set{g \in \pi^{-1}(\rho(Y)) }{ \hat{h}_G(g) < \epsilon } \, . 
	\]
	However, this follows immediately from Remark~\ref{Rem.Height}.
\end{proof}

We now make the connection to Breuillard's height $\hat{h}_{B}$ \cite[p.1058 and Sect.~2.2]{breuillardgap}. 
For this we assume that $G = \hat{G} = \GL_t(\overline{\QQ})$  and that 
$\hat{S} = \mathbb{G}_m^t$ is the diagonal maximal torus. According to \cite[p.1058]{breuillardgap} we have that 
\[
	\hat{h}_{B}(g) = h_t([1:\lambda_1:\cdots:\lambda_t]) \, ,
\]
where $(\lambda_1,  \dots,  \lambda_t)$ is a tuple of eigenvalues of $g$ and $h_n \colon \PP^n \to \RR_{\geq 0}$ 
is the Weil-height on $\mathbb{P}^n$ given by 
\begin{equation}
	\label{wheight} 
	h_n([\mu_0: \mu_1: \cdots : \mu_n]) = \frac1{[K:\mathbb{Q}]} \sum_{v \in M_K} n_v
	\log\max\{|\mu_0|_v,|\mu_1|_v,  \dots, |\mu_n|_v\},
\end{equation}
where $K$ is a number field containing $\mu_0, \dots, \mu_n$, $M_K$ is the set of absolute values of $K$, normalized as in the 
citation, $n_v = [K_v:\Q_v]$ and $K_v$, $\Q_v$ are the completions with respect to $v$. 
Note that  $h_t([1:\lambda_1:\cdots:\lambda_t])$ does not depend on  the ordering of the elements in the tuple $(\lambda_1, \dots, \lambda_t)$. 

The Weyl-group of $\GL_t(\overline{\QQ})$ with respect to $\mathbb{G}_m^t$ is equal to the symmetric group
$\mathcal{S}_t$. The divisor at infinity $D_{\infty}$ of $(\PP^1)^t$ induced the composition of the
natural closed embedding $\iota \colon (\PP^1)^t \to \PP^{2^t-1}$ coming from the Segre embedding
with the Veronese embedding $\eta \colon \PP^{2^t-1} \to \PP^{2^{t+1}-1}$ of degree $2$.
Note that $h_{2^{t+1}-1} \circ \eta \circ \iota = 2 \cdot \sum_{i=1}^t h_1 \circ \pr_i$, where $\pr_i \colon (\PP^1)^t \to \PP^1$
denotes the projection to the $i$-th factor, see~\cite[Proposition~B.2.4]{hindrysilverman}.
Using that $|h_{D_{\infty}}- h_{2^{t+1}-1} \circ \eta \circ \iota|$ is bounded (\cite[Theorem~B.3.2(a),(b)]{hindrysilverman})
and that $h_1([1: \lambda]) =  \lim_{k \to \infty} \frac{h_1([2^k] [1: \lambda])}{2^k}$ for all 
$\lambda \in \overline{\QQ}^\ast$, we get that our height $\hat{h}_G$ on $G = \GL_t(\overline{\QQ})$ is given by 
\[
	\hat{h}_G(g) = \sum_{\sigma \in \mathcal{S}_t} \lim_{k \to \infty}
	\frac{h_{D_{\infty}}([2^k](\lambda_{\sigma(1)}, \ldots, \lambda_{\sigma(t)}) )}{2^k} = 2 \cdot (t!) \cdot \sum_{i = 1}^th_1([1 :\lambda_i]) \, .
\]
Now we do have for all $g \in G$ the inequalities 
\begin{equation}
	\label{hG}
	\frac1{2 \cdot (t!)}\hat{h}_G(g) \leq t \hat{h}_{B}(g) \leq \hat{h}_G(g) \, ,
\end{equation}
and so for $G = \GL_t(\overline{\QQ})$, Theorem \ref{bog}  also holds with $\hat{h}_{G}$ replaced by
Breuillard's height $\hat{h}_B$.  

\providecommand{\bysame}{\leavevmode\hbox to3em{\hrulefill}\thinspace}
\providecommand{\MR}{\relax\ifhmode\unskip\space\fi MR }
\providecommand{\MRhref}[2]{%
	\href{http://www.ams.org/mathscinet-getitem?mr=#1}{#2}
}
\providecommand{\href}[2]{#2}


\begin{thebibliography}{CRRZ22}
	
	\bibitem[BD21]{barroero2021distinguished}
	Fabrizio Barroero and Gabriel~Andreas Dill, \emph{Distinguished categories and
		the {Z}ilber-{P}ink conjecture (ar{X}iv)}, 2021.
	
	\bibitem[Bor91]{Bo1991Linear-algebraic-g}
	Armand Borel, \emph{Linear algebraic groups}, second ed., Graduate Texts in
	Mathematics, vol. 126, Springer-Verlag, New York, 1991.
	
	\bibitem[Bre11]{breuillardgap}
	Emmanuel Breuillard, \emph{A height gap theorem for finite subsets of {${\rm
				GL}_d(\overline{\Bbb Q})$} and nonamenable subgroups}, Ann. of Math. (2)
	\textbf{174} (2011), no.~2, 1057--1110. \MR{2831113}
	
	\bibitem[Bri09]{Br2009Anti-affine-algebr}
	Michel Brion, \emph{Anti-affine algebraic groups}, J. Algebra \textbf{321}
	(2009), no.~3, 934--952. \MR{2488561}
	
	\bibitem[CRRZ22]{boundedlygen}
	Pietro Corvaja, Andrei~S. Rapinchuk, Jinbo Ren, and Umberto~M. Zannier,
	\emph{Non-virtually abelian anisotropic linear groups are not boundedly
		generated}, Invent. Math. \textbf{227} (2022), no.~1, 1--26. \MR{4359474}
	
	\bibitem[DG70]{DeGa1970Groupes-algebrique}
	Michel Demazure and Pierre Gabriel, \emph{Groupes alg\'{e}briques. {T}ome {I}:
		{G}\'{e}om\'{e}trie alg\'{e}brique, g\'{e}n\'{e}ralit\'{e}s, groupes
		commutatifs}, Masson \& Cie, \'{E}diteurs, Paris; North-Holland Publishing
	Co., Amsterdam, 1970, Avec un appendice {{\i}t Corps de classes local} par
	Michiel Hazewinkel. \MR{0302656}
	
	\bibitem[DP00]{davidphilippon}
	Sinnou David and Patrice Philippon, \emph{Torsion subvarieties of semi-abelian
		varieties}, C. R. Acad. Sci., Paris, S{\'e}r. I, Math. \textbf{331} (2000),
	no.~8, 587--592 (French).
	
	\bibitem[Fei96]{feitorders}
	Walter Feit, \emph{Orders of finite linear groups}, Proceedings of the {F}irst
	{J}amaican {C}onference on {G}roup {T}heory and its {A}pplications
	({K}ingston, 1996), Univ. West Indies, Kingston, [1996], pp.~9--11.
	\MR{1484185}
	
	\bibitem[Fri97]{maximalorders}
	Shmuel Friedland, \emph{The maximal orders of finite subgroups in {${\rm
				GL}_n({\bf Q})$}}, Proc. Amer. Math. Soc. \textbf{125} (1997), no.~12,
	3519--3526. \MR{1443385}
	
	\bibitem[GP93]{GrPf1993Geometric-Quotient}
	Gert-Martin Greuel and Gerhard Pfister, \emph{{Geometric Quotients of Unipotent
			Group Actions}}, Proceedings of the London Mathematical Society
	\textbf{s3-67} (1993), no.~1, 75--105.
	
	\bibitem[GP11]{2011Schemas-en-groupes}
	Philippe Gille and Patrick Polo (eds.), \emph{Sch\'{e}mas en groupes ({SGA} 3).
		{T}ome {I}. {P}ropri\'{e}t\'{e}s g\'{e}n\'{e}rales des sch\'{e}mas en
		groupes}, Documents Math\'{e}matiques (Paris) [Mathematical Documents
	(Paris)], vol.~7, Soci\'{e}t\'{e} Math\'{e}matique de France, Paris, 2011,
	S\'{e}minaire de G\'{e}om\'{e}trie Alg\'{e}brique du Bois Marie 1962--64.
	[Algebraic Geometry Seminar of Bois Marie 1962--64], A seminar directed by M.
	Demazure and A. Grothendieck with the collaboration of M. Artin, J.-E.
	Bertin, P. Gabriel, M. Raynaud and J-P. Serre, Revised and annotated edition
	of the 1970 French original. \MR{2867621}
	
	\bibitem[Gro61]{Gr1961Elements-de-geomet-III}
	Alexander Grothendieck, \emph{\'{E}l{\'e}ments de g{\'e}om{\'e}trie
		alg{\'e}brique. {III}. \'{E}tude cohomologique des faisceaux coh{\'e}rents.
		{I}}, Inst. Hautes {\'E}tudes Sci. Publ. Math. (1961), no.~11, 167.
	
	\bibitem[GW10]{GoWe2010Algebraic-geometry}
	Ulrich G\"{o}rtz and Torsten Wedhorn, \emph{Algebraic geometry {I}}, Advanced
	Lectures in Mathematics, Vieweg + Teubner, Wiesbaden, 2010, Schemes with
	examples and exercises.
	
	\bibitem[Hin88]{Hi1988Autour-dune-conjec}
	Marc Hindry, \emph{Autour d'une conjecture de serge lang}, Inventiones
	mathematicae \textbf{94} (1988), no.~3, 575--603.
	
	\bibitem[HS00]{hindrysilverman}
	Marc Hindry and Joseph~H. Silverman, \emph{Diophantine geometry. {An}
		introduction}, Grad. Texts Math., vol. 201, New York, NY: Springer, 2000
	(English).
	
	\bibitem[Hum75]{Hu1975Linear-algebraic-g}
	James~E. Humphreys, \emph{Linear algebraic groups}, Springer-Verlag, New
	York-Heidelberg, 1975, Graduate Texts in Mathematics, No. 21.
	
	\bibitem[Hum95]{Hu1995Conjugacy-classes-}
	\bysame, \emph{Conjugacy classes in semisimple algebraic groups}, Mathematical
	Surveys and Monographs, vol.~43, American Mathematical Society, Providence,
	RI, 1995. \MR{1343976}
	
	\bibitem[Kra84]{Kr1984Geometrische-Metho}
	Hanspeter Kraft, \emph{Geometrische {M}ethoden in der {I}nvariantentheorie},
	Aspects of Mathematics, D1, Friedr. Vieweg \& Sohn, Braunschweig, 1984.
	\MR{768181}
	
	\bibitem[K{\"u}h20]{KuehneBHC}
	Lars K{\"u}hne, \emph{The bounded height conjecture for semiabelian varieties},
	Compos. Math. \textbf{156} (2020), no.~7, 1405--1456. \MR{4120167}
	
	\bibitem[K{\"u}h22]{kuehnesmall}
	\bysame, \emph{Points of small height on semiabelian varieties}, J. Eur. Math.
	Soc. (JEMS) \textbf{24} (2022), no.~6, 2077--2131.
	
	\bibitem[Maz77]{Mazurrational}
	B.~Mazur, \emph{Modular curves and the {E}isenstein ideal}, Inst. Hautes
	\'{E}tudes Sci. Publ. Math. (1977), no.~47, 33--186 (1978), With an appendix
	by Mazur and M. Rapoport. \MR{488287}
	
	\bibitem[Ost]{ostafeerratum}
	A.~Ostafe, \emph{Erratum to ``{O}n a {P}roblem of {L}ang for {M}atrix
		{P}olynomials"}, Homepage of A. Ostafe
	``https://web.maths.unsw.edu.au/~alinaostafe/".
	
	\bibitem[Ray83a]{Raynaudcourbes}
	M.~Raynaud, \emph{Courbes sur une vari\'{e}t\'{e} ab\'{e}lienne et points de
		torsion}, Invent. Math. \textbf{71} (1983), no.~1, 207--233. \MR{688265}
	
	\bibitem[Ray83b]{Raynaudsous}
	\bysame, \emph{Sous-vari\'{e}t\'{e}s d'une vari\'{e}t\'{e} ab\'{e}lienne et
		points de torsion}, Arithmetic and geometry, {V}ol. {I}, Progr. Math.,
	vol.~35, Birkh\"{a}user Boston, Boston, MA, 1983, pp.~327--352. \MR{717600}
	
	\bibitem[Ros56]{Ro1956Some-basic-theorem}
	Maxwell Rosenlicht, \emph{Some basic theorems on algebraic groups}, Amer. J.
	Math. \textbf{78} (1956), 401--443. \MR{82183}
	
	\bibitem[Ros58]{Ro1958Extensions-of-vect}
	\bysame, \emph{Extensions of vector groups by abelian varieties}, Amer. J.
	Math. \textbf{80} (1958), 685--714. \MR{99340}
	
	\bibitem[Ser58]{Se1958Espaces-fibres-alg}
	Jean-Pierre Serre, \emph{Espaces fibr\'es alg\'ebriques}, S\'eminaire Claude
	Chevalley \textbf{3} (1958) (fr), talk:1.
	
	\bibitem[SR17]{SaRi2017Actions-and-invari}
	Walter~Ferrer Santos and Alvaro Rittatore, \emph{Actions and invariants of
		algebraic groups}, second ed., Monographs and Research Notes in Mathematics,
	CRC Press, Boca Raton, FL, 2017. \MR{3617213}
	
	\bibitem[Ste65]{St1965Regular-elements-o}
	Robert Steinberg, \emph{Regular elements of semisimple algebraic groups}, Inst.
	Hautes \'{E}tudes Sci. Publ. Math. (1965), no.~25, 49--80. \MR{180554}
	
	\bibitem[Ste74]{St1974Conjugacy-classes-}
	\bysame, \emph{Conjugacy classes in algebraic groups}, Lecture Notes in
	Mathematics, Vol. 366, Springer-Verlag, Berlin-New York, 1974, Notes by Vinay
	V. Deodhar. \MR{0352279}
	
	\bibitem[Tal00]{talamancaheight}
	Valerio Talamanca, \emph{A {G}elfand-{B}eurling type formula for heights on
		endomorphism rings}, J. Number Theory \textbf{83} (2000), no.~1, 91--105.
	\MR{1767654}
	
	\bibitem[Tze00]{Tz2000The-Manin-Mumford-}
	Pavlos Tzermias, \emph{The {M}anin-{M}umford conjecture: a brief survey}, Bull.
	London Math. Soc. \textbf{32} (2000), no.~6, 641--652. \MR{1781574}
	
	\bibitem[Ull98]{ullmobogomolov}
	Emmanuel Ullmo, \emph{Positivit\'{e} et discr\'{e}tion des points
		alg\'{e}briques des courbes}, Ann. of Math. (2) \textbf{147} (1998), no.~1,
	167--179. \MR{1609514}
	
	\bibitem[Yaf22]{andreisubspace}
	Andrei Yafaev, \emph{Non-commutative analytic subgroup theorem}, J. Number
	Theory \textbf{230} (2022), 233--237. \MR{4327956}
	
	\bibitem[Zha98]{zhangbogomolov}
	Shou-Wu Zhang, \emph{Equidistribution of small points on abelian varieties},
	Ann. of Math. (2) \textbf{147} (1998), no.~1, 159--165. \MR{1609518}
	
\end{thebibliography}
\end{document}